\documentclass[11pt,letterpaper]{article}
\usepackage[T1]{fontenc}
\usepackage[utf8]{inputenc}
\usepackage{libertinus-type1}
\usepackage{microtype}
\usepackage[margin=1.15in]{geometry}
\usepackage{amsmath,amssymb,amsthm,mathtools,mathrsfs}
\usepackage{libertinust1math}
\usepackage{enumitem}
\usepackage{needspace}
\usepackage[dvipsnames]{xcolor}
\definecolor{paperlink}{HTML}{8B0000}
\definecolor{papercite}{HTML}{00CED1}
\definecolor{paperurl}{HTML}{EC008C}
\usepackage[nosort,nocompress]{cite}

\usepackage[colorlinks=true,linkcolor=paperlink,citecolor=papercite,urlcolor=paperurl]{hyperref}
\usepackage[nameinlink,noabbrev]{cleveref}
\numberwithin{equation}{section}
\newtheorem{theorem}{Theorem}[section]
\newtheorem{lemma}[theorem]{Lemma}
\newtheorem{corollary}[theorem]{Corollary}
\newtheorem{proposition}[theorem]{Proposition}
\theoremstyle{remark}

\newcommand{\R}{\mathbb R}

\newcommand{\E}{\mathbb E}
\renewcommand{\P}{\mathbb P}
\newcommand{\Sstrip}{\mathcal S}
\newcommand{\dd}{\,\mathrm d}
\newcommand{\ind}{\mathbf 1}

\DeclareMathOperator{\sech}{sech}
\DeclareMathOperator{\Var}{Var}
\DeclareMathOperator{\Cov}{Cov}
\DeclareMathOperator{\supp}{supp}
\allowdisplaybreaks[2]
\hypersetup{pdftitle={The semiclassical limit of unit-area quantum disks },pdfauthor={},pdfsubject={Confidential manuscript}}
\title{The semiclassical limit of unit-area quantum disks}
\author{
  Yuchen Fan\thanks{These authors contributed equally
  and are joint first authors.}\\
  {\small Courant Institute, New York University, USA}\\
  {\small yf2085@nyu.edu}
  \and
  Zhenfeng Tu\footnotemark[1]\\
  {\small Courant Institute, New York University, USA}\\
{\small zt2255@nyu.edu}
}
\date{}
\begin{document}
\maketitle
\begin{abstract}
Fix $W\in(0,2)$ and let $\gamma\downarrow0$. We prove that a
weight-$W$ quantum disk conditioned to have unit quantum area
converges, in the maximum-embedded strip, to the deterministic field
$\phi_W(t+i\theta)=\log(W/(4\pi))-2\log\cosh(Wt/2)$ and its
constant-curvature area measure. The proof establishes a joint
large-deviation principle for the radial Brownian input, the lateral
Gaussian free field, and the total Gaussian multiplicative chaos
mass. Diverging-order moment bounds justify both adjoining the
noncontinuous area functional and retaining its unbounded positive
power in the unit-area expectation. An exact radial Girsanov shift
centers the candidate, and reflection together with the sharp Onofri
inequality proves its unique optimality. This yields concentration
without a quadratic expansion of the running maximum or an order-one
partition-function asymptotic. We also identify the field rate with
the renormalized Liouville action in the maximum embedding.
\end{abstract}


\section{Introduction}
\label{intro:section}

We study the semiclassical limit of a quantum disk of fixed weight
$W\in(0,2)$ conditioned to have quantum area one. The limiting surface
has positive constant curvature and two marked boundary points.
Our proof evaluates the normalized unit-area expectation by a joint
large-deviation argument: we retain the radial field, the lateral
field, and the total area as simultaneous variables, then identify
the unique profile selected by the area weight. Two features make
this formulation useful. The area is a Gaussian multiplicative chaos
mass, hence is not continuous in the topology of the field; and the
maximum embedding introduces a Brownian running maximum which does
not admit a direct quadratic small-noise expansion. We handle the
first feature by exponentially good approximations and the second
by keeping the reflection map exact throughout the variational
analysis.

The quantum disk is a basic finite-volume surface in Liouville
quantum gravity and in the mating-of-trees construction
\cite{DMS21,AG21}. Liouville theory on simply connected surfaces
with boundary was constructed in \cite{HRV18}; the equivalence of
fixed-boundary-length descriptions is established in \cite{Cer21},
and the weight-$W$ disks used in conformal welding are treated in
\cite{AHS23}. We use their maximum-embedded strip description.
Section~\ref{setup:section} gives the precise conventions for the
field, its area measure, and conditioning on area one. Here we state
the result and explain the proof, including the estimates needed to
pass from the Gaussian law to the area-weighted expectation.

\subsection{The limit and the exact area weight}

The starting point is an exact expression for unit-area expectation.
It exhibits an area reward at the same scale as the Gaussian
large-deviation cost, so the limit must be found by optimizing their
difference. We first describe that identity and the profile selected
by the optimization.

Write $\Sstrip=\R\times(0,\pi)$, with $z=t+i\theta$ and planar
Lebesgue measure $d^2z=dt\,d\theta$. For a Liouville field $h$, its
quantum area is the regularized random measure
\[
 \mu_h^\gamma(d^2z)=\lim_{\varepsilon\downarrow0}
       \varepsilon^{\gamma^2/2}e^{\gamma h_\varepsilon(z)}d^2z,
\]
where $h_\varepsilon$ is the circle average; the limit is understood
locally in probability in the weak topology of measures.
Let $\widehat h=R_\gamma+X^\perp$ be the disk field before adding
its zero mode. The lateral field $X^\perp$ is the zero-angular-mean
projection of the Neumann GFF, independent of $R_\gamma$. On either
half-line the radial process is variance-two Brownian motion with
drift $-W/\gamma+\gamma/2$, conditioned to stay nonpositive and
started from its entrance point zero; the two halves are independent.
Thus the vertical average of $\widehat h$ has maximum zero at $t=0$.
This fixes horizontal translation and is called the maximum embedding.

Set $Z_\gamma=\mu_{\widehat h}^\gamma(\Sstrip)$ and denote by
$h_\gamma^{(1)}$ the unit-area field. Disintegrating the zero mode
gives, for bounded Borel field observables $F$,
\begin{equation}\label{intro:ratio}
 \E^{(1)}[F(\gamma h_\gamma^{(1)})]
 =\frac{\E\!\left[
 F(\gamma\widehat h-\log Z_\gamma)
 Z_\gamma^{p_\gamma}\right]}
 {\E[Z_\gamma^{p_\gamma}]},
 \qquad p_\gamma=\frac W{\gamma^2}-\frac12.
\end{equation}
Here $\P$ and $\E$ refer to the independent radial and lateral
fields before the area tilt. The same mass appears both in the
weight and in the additive correction enforcing unit area.
Since $\gamma^2p_\gamma\to W$, a mass $a$ contributes the
logarithmic reward $W\log a$ at speed $\gamma^{-2}$. Typical behavior
under $\P$ alone therefore does not determine the unit-area limit.

Theorem~\ref{setup:main} proves that the selected profile is
\begin{equation}\label{intro:classical}
 \phi_W(t+i\theta)=\log\frac W{4\pi}
                        -2\log\cosh\frac{Wt}{2}.
\end{equation}
For every $s>0$ and every bounded continuous real functional
$F$ on $H^{-s}_{\mathrm{loc}}(\Sstrip)$,
\begin{equation}\label{intro:target}
 \lim_{\gamma\downarrow0}\E^{(1)}[F(\gamma h_\gamma^{(1)})]
                         =F(\phi_W).
\end{equation}
The unit-area measures converge weakly in probability to
$e^{\phi_W}d^2z$, including on the compactified strip, without mass
escaping to the ends or the angular boundary. The limiting metric
has area one, curvature $\pi W$, geodesic boundary away from the
marked points, and opening angle $\pi W/2$ at each marked point.
The corresponding constant-curvature geometry is consistent with
the classification of disks with two boundary corners in \cite{JWZ09}.

\subsection{One joint LDP for the fields and their area}

Section~\ref{decomp:section} proves the probabilistic input once,
allowing any fixed finite-energy radial displacement. Its main
milestone is Lemma~\ref{ldp:joint}: the joint rate is exactly the
Gaussian energy on the graph of the deterministic area. This result
both accounts for the dependence of the mass on the fields and
permits the later Girsanov change of variables without another LDP
proof.

Let $B$ have independent variance-two Brownian halves, pinned at
zero and independent of $X^\perp$. For $t\ge0$ and $\sigma\in\{-1,1\}$,
define
\[
 d_\gamma(u;\sigma t)=u(\sigma t)-(W-\gamma^2/2)t
 -2\sup_{0\le r\le t}\{u(\sigma r)-(W-\gamma^2/2)r\}.
\]
The drifted Pitman representation, Lemma~\ref{setup:Pitman}, gives
$\gamma R_\gamma\overset{\rm law}=d_\gamma(\gamma B)$.
The map $d_\gamma$ is continuous locally in time, is nonpositive,
and converges locally uniformly to $d_0$, uniformly in its input.
The Gaussian pair $(\gamma B,\gamma X^\perp)$ has a good LDP at
speed $\gamma^{-2}$ with energy
\[
 I_G(u,v)=\frac14\int_\R|u'|^2\,dt
            +\frac1{4\pi}\int_{\Sstrip}|\nabla v|^2\,d^2z.
\]
Finite energy means that $u$ is locally absolutely continuous,
$u(0)=0$, $u'\in L^2(\R)$, and $v\in H^1(\Sstrip)$ has zero
angular mean. The Gaussian topology and compactness of these energy
sublevels are established in Lemma~\ref{ldp:gaussian}.

For a deterministic radial exponent $f$, write
\[
 Z_\gamma(f)=\int_{\Sstrip}e^{f(t)}
       (2\sin\theta)^{-\gamma^2/2}:e^{\gamma X^\perp(t,\theta)}:
                 \,dt\,d\theta.
\]
Colons mean full Gaussian variance subtraction; the displayed angular
factor converts this Wick normalization to the canonical quantum
area. The integral is a GMC mass, not a pointwise exponential of a
distribution. Subsection~\ref{reg:section} identifies its Fourier
construction with the canonical circle-average measure on the same
probability space.
For any fixed locally absolutely continuous $g$ with $g(0)=0$ and
$g'\in L^2(\R)$, Lemma~\ref{ldp:joint} gives the joint LDP under
$\P$ for
\[
 (\gamma B,\gamma X^\perp,Z_\gamma(d_\gamma(\gamma B+g)))
\]
with rate
\[
 I_g(u,v,a)=
 \begin{cases}
 I_G(u,v),&a=\displaystyle\int_{\Sstrip}
                  e^{d_0(u+g;t)+v(t,\theta)}\,dt\,d\theta,
                  \quad I_G(u,v)<\infty,\\
 +\infty,&\text{otherwise}.
 \end{cases}
\]
There is no independent area cost: the finite-rate area is determined
by the same radial and lateral profiles whose energy is paid.

The essential step is to justify this graph constraint despite the
noncontinuity of GMC. On a fixed box $[-T,T]\times(0,\pi)$, projection
onto the first $N$ angular modes makes the deterministic mass a
continuous function of the Gaussian pair. Lemma~\ref{mom:cutoff-error}
proves that removing the Fourier cutoff has arbitrarily negative
logarithmic probability rate. Proposition~\ref{mom:whole-strip},
together with a Gaussian tail bound for the radial norm, does the same
for removing the box. The Wick and boundary corrections at fixed
cutoffs vanish uniformly on compact Gaussian sets. In parallel,
the cell consequence of Onofri, Lemma~\ref{var:cell-exponential},
gives uniform deterministic approximation on every energy sublevel.
Choosing the boxes first and then the mode cutoffs produces one
exponentially good approximation sequence, to which the extended
contraction principle applies.

The estimates work at the diverging moment order
$q_\gamma=1+(W+2)/(2\gamma^2)$. On bounded boxes, the Fourier error
tends to zero in $L^{q_\gamma}(\mathbb P)$, and Markov's inequality
turns this into a superexponential probability estimate because
$\gamma^2q_\gamma\to(W+2)/2>0$. The same moment bounds will control
the area reward, since $q_\gamma>p_\gamma$ with a strict relative
margin for $W<2$. Their proof retains horizontal Lebesgue integration
along with the angular coordinate; this two-dimensional collision
estimate is what permits the full range $0<W<2$.

\subsection{The weighted variational problem and concentration}

Section~\ref{rad:section} turns the joint probability LDP into the
limit of \eqref{intro:ratio}. An exact radial Girsanov shift centers
the candidate; reflection and Onofri identify its unique optimality;
and high moments justify weighted Laplace bounds. These steps yield
concentration under the normalized weight, with the observable and
its area normalization retained throughout.

The radial stationarity calculation in Subsection~\ref{proof:shift}
selects
\[
 q_W(t)=-2\log\cosh(Wt/2),\qquad
 g_*(t)=W|t|+q_W(t),\qquad
 \int_{\Sstrip}e^{q_W(t)}\,d^2z=\frac{4\pi}{W}.
\]
Translate the Brownian input by $g_*/\gamma$, and put
$b_\gamma(u)=d_\gamma(u+g_*)$,
$A_\gamma^*=Z_\gamma(b_\gamma(\gamma B))$, and
\[
 K(u)=\frac{W^2}{4}\int_\R u(t)\sech^2(Wt/2)\,dt,
 \qquad
 H_\gamma=\left(\frac{WA_\gamma^*}{4\pi}\right)^{p_\gamma}
                        e^{-K(\gamma B)/\gamma^2}.
\]
Lemma~\ref{setup:shifted-identity} gives the exact identity
\[
 \E^{(1)}[F(\gamma h_\gamma^{(1)})]
 =\frac{\E[F(b_\gamma(\gamma B)+\gamma X^\perp
                              -\log A_\gamma^*)H_\gamma]}
        {\E H_\gamma}.
\]
The expectations on the right still use the centered Gaussian law
$\P$. The deterministic Cameron--Martin factor cancels from the
ratio. We denote the probability law with density
$H_\gamma/\E H_\gamma$ by $\mathbb Q_\gamma$.

Apply the unified joint LDP with $g=g_*$. If
$A_0^*(u,v)=\int e^{b_0(u)+v}\,d^2z$, the logarithmic weight minus
the joint rate, on its finite-rate graph, is
\begin{align}\label{intro:variational-expression}
 -J(u,v)
 ={}&W\log\left[\frac W{4\pi}A_0^*(u,v)\right]-K(u)\notag\\
 &-\frac14\int_\R|u'|^2\,dt
   -\frac1{4\pi}\int_{\Sstrip}|\nabla v|^2\,d^2z.
\end{align}
The use of $A_0^*$ here follows from the joint rate's graph constraint;
it does not assume that the random mass can already be replaced by a
constant inside the expectation.

The variational milestone is Lemma~\ref{var:selection}. For
$m_\sigma(t)=\sup_{0\le r\le t}\{q_W(r)+u(\sigma r)\}$, reflection
converts the radial energy and source into those of
$f(\sigma t)=u(\sigma t)-2m_\sigma(t)$, with an explicit cost for
the record heights. Applying the sharp strip Onofri inequality to
$f+v$ gives
\[
 J(u,v)\ge W\sum_{\sigma=\pm}m_\sigma(\infty)
       +\frac{4-W^2}{16\pi}\int_{\Sstrip}|v_\theta|^2\,d^2z
       \ge0.
\]
For $W<2$, equality forces $v=0$; the radial equality case and the
maximum constraint then force $u=0$. A separate coercivity estimate,
Lemma~\ref{var:coercivity}, gives compact sublevels of $J$ and hence
a strict gap away from this optimizer. The strict inequality $W<2$
thus enters both the moment estimates and uniqueness in the
variational problem.

To pass from this comparison to expectations, Lemma~\ref{ldp:weight-moment}
proves, for some fixed $q>1$,
\[
 \limsup_{\gamma\downarrow0}\gamma^2\log\E[H_\gamma^q]<\infty.
\]
Together with exponential tightness, this removes the weighted
contribution outside compact sets. On compact sets bounded away from
zero area, the logarithmic weight converges uniformly; small areas
are controlled separately. Lemma~\ref{ldp:weighted-upper} consequently
gives upper and lower weighted Laplace bounds with objective $-J$.
They imply
\[
 \gamma^2\log\E H_\gamma\longrightarrow0,\qquad
 (\gamma B,\gamma X^\perp,A_\gamma^*)
        \longrightarrow(0,0,4\pi/W)
        \quad\text{under }\mathbb Q_\gamma,
\]
with exponentially small probabilities outside any neighborhood of
the limiting triple. Continuity of the field argument at this triple
and boundedness of $F$ prove \eqref{intro:target}. Adjoining finitely
many tested masses by the same approximation argument gives area
convergence; total mass one prevents escape of area. Keeping the
deterministic Girsanov factor also yields the free energy
\[
 \lim_{\gamma\downarrow0}\gamma^2\log\E[Z_\gamma^{p_\gamma}]
                  =W\left(1+\log\frac\pi W\right).
\]
No order-one asymptotic for $\E H_\gamma$ is needed for these conclusions.

\subsection{Comparison with the semiclassical expansion of Lacoin--Rhodes--Vargas}

We explain why we work at large-deviation scale rather than seek a
quadratic expansion of the shifted expectation. The issue is both the
positive area power and the running maximum in the maximum embedding.
The latter produces an intermediate scale near the origin, which
requires separate analysis if one wants fluctuation asymptotics.

Lacoin--Rhodes--Vargas \cite{LRV22} treat the negative-curvature
semiclassical regime on the sphere and obtain partition-function
asymptotics, Gaussian field fluctuations, and an LDP. Their
Cameron--Martin shift is followed by a second-order expansion of the
logarithmic chaos mass; convergence of the residual integrand and
uniform integrability justify passage to its expectation
\cite[Sections 4.5 and 5, Proposition 4.3]{LRV22}.
They also discuss fixed-volume positive curvature and identify the
opposite sign of the exponential-moment estimate as an obstruction
to extending that argument \cite[Section 2.3]{LRV22}.
Our unit-area disk has the positive power in \eqref{intro:ratio}.
We control it through diverging-order positive moments and weighted
Laplace bounds, which suffice for concentration without an order-one
limit of the residual integrand.

There is a further difficulty specific to our radial representation.
Under the shifted Gaussian reference law $\P$, the running maximum
on the positive half-line is
\[
 M_\gamma(t)=\sup_{0\le r\le t}
       \{\gamma B(r)-2\log\cosh(Wr/2)+\gamma^2r/2\},
 \qquad
 b_\gamma(\gamma B;t)=q_W(t)+\gamma B(t)+\gamma^2t/2
                                     -2M_\gamma(t).
\]
The deterministic profile has a quadratic maximum at zero:
$q_W(r)=-W^2r^2/4+O(r^4)$. Brownian noise there has size
$\gamma\sqrt r$, so balancing it against $r^2$ gives a time scale
$r\asymp\gamma^{2/3}$ and a height scale $\gamma^{4/3}$.
More precisely, for every fixed $L>0$, Brownian scaling and the
uniform Taylor expansion on $[0,L]$ give
\begin{equation}\label{intro:maximum-scale}
 \gamma^{-4/3}M_\gamma(\gamma^{2/3}L)
 \ \Longrightarrow\
 \sup_{0\le s\le L}\{\widetilde B(s)-W^2s^2/4\},
\end{equation}
where $\widetilde B$ is variance-two Brownian motion. Indeed, the
process inside the rescaled supremum has the law of
$\widetilde B(s)+\gamma^{-4/3}q_W(\gamma^{2/3}s)
+\gamma^{4/3}s/2$, whose deterministic part converges uniformly to
$-W^2s^2/4$. The limiting maximum is strictly positive almost surely.
In particular, the running maximum is not an $O_{\mathbb P}(\gamma^2)$
remainder: for any fixed $t>0$, monotonicity and
\eqref{intro:maximum-scale} imply
$M_\gamma(t)/\gamma^2\to\infty$ in probability under $\P$.

Consequently a direct expansion of the reflected radial field through
order $\gamma^2$ must account for the maximum separately. A correction
of height $\gamma^{4/3}$, when multiplied by the exponent of order
$\gamma^{-2}$ in the area weight, can contribute at the intermediate
scale $\gamma^{-2/3}$. The smooth quadratic chaos expansion alone
does not control this term. Our proof avoids that expansion: the
running maximum remains inside the continuous radial map and is
handled exactly by the reflection identity in the rate function.

This observation does not establish failure of a field fluctuation
limit. The centered unscaled field is
\[
 h_\gamma^{(1)}-\gamma^{-1}\phi_W
       =\gamma^{-1}(\gamma h_\gamma^{(1)}-\phi_W),
\]
so order-one field fluctuations correspond to order $\gamma$, not
order $\gamma^2$, in the scaled field. A maximum correction of size
$\gamma^{4/3}$ would contribute only $\gamma^{1/3}$ at that scale.
Moreover, \eqref{intro:maximum-scale} concerns $\P$; the unit-area
weight changes the law to $\mathbb Q_\gamma$ and can alter these local
scales. Obtaining a fluctuation theorem would require analyzing that
weighted local behavior together with the remaining field and mass.
Uniform integrability could justify such a limit once convergence of
the relevant residual integrand is proved; it does not by itself
supply the fluctuation expansion. We prove concentration and large
deviations here and leave this finer question open.

Section~\ref{setup:section} specifies the quantum disk and its
unit-area law. Section~\ref{decomp:section} proves the unified joint
LDP and Section~\ref{rad:section} applies it to the weighted
expectation as described above. Finally, Appendix~\ref{act:section}
contracts the tilted joint rate to the normalized field and identifies
it with the renormalized Liouville action in the maximum embedding.

\section{Quantum disks, unit-area conditioning, and statement of the result}
\label{setup:section}

This section keeps separate the four ingredients that enter the definition:
the conformal notion of a quantum surface, the Wick-ordering convention,
the maximum-embedded weight-$W$ disk, and the disintegration at unit area.
The distinction is useful because only the lateral Gaussian field is Wick
ordered, whereas the radial process and the zero mode are treated by their
ordinary exponential weights.

We use $\P$ and $\E$ for probability and expectation before the
unit-area tilt, and $\ind_A$ for the indicator of an event $A$.
For $1\le p<\infty$, $\|Y\|_{L^p}=(\E|Y|^p)^{1/p}$ is the
probability norm; a spatial $L^p$ norm uses the stated reference measure.
The notation $\lesssim_{W,p}$ means that the ratio is bounded by a
finite positive number depending only on the displayed subscripts;
$O_{W,p}(1)$ denotes a quantity with the same bound in absolute value.
An omitted subscript means an absolute bound.
We write $C_c$ for continuous functions of compact support and
$C_c^\infty$ for smooth functions of compact support; $\supp f$
denotes the closed support of $f$. The notation $\Var$ and $\Cov$
means variance and covariance. We set
$\log^+x=\max\{0,\log x\}$, $x\wedge y=\min\{x,y\}$, and
$x\vee y=\max\{x,y\}$. The symbol $\delta_0$ denotes the Dirac
mass at zero. All sufficiently-small
$\gamma$ statements are for fixed $0<W<2$.

\subsection{Quantum surfaces and quantum area}

We first give the precise definition of the quantum disk used in this
paper.  A doubly marked $\gamma$-Liouville quantum surface is an equivalence
class of tuples
\[
 (D,h,z_-,z_+),
\]
where $D$ is a simply connected domain, $z_-,z_+\in\partial D$ are distinct
marked boundary points, and $h$ is a distribution on $D$.  If
$f:\widetilde D\to D$ is conformal, then
\begin{equation}\label{setup:surface-equivalence}
 (D,h,z_-,z_+)\sim
 \left(\widetilde D,
 h\circ f+\left(\frac2\gamma+\frac\gamma2\right)\log|f'|,
 f^{-1}(z_-),f^{-1}(z_+)\right).
\end{equation}
The corresponding quantum area measure is the Gaussian multiplicative
chaos measure
\begin{equation}\label{setup:circle-area}
 \mu_h^\gamma(d^2z)
 =\lim_{\varepsilon\downarrow0}
               \varepsilon^{\gamma^2/2}e^{\gamma h_\varepsilon(z)}d^2z,
\end{equation}
Here $h_\varepsilon(z)$ is the average of $h$ on the Euclidean circle
of radius $\varepsilon$ about $z$, for $\varepsilon<
\operatorname{dist}(z,\partial D)$. The limit is taken against
compactly supported continuous tests, in probability. For the disk
fields constructed below it is a positive Radon measure. When another
regularization is used, we include the explicit finite-variance
correction described in Subsection~\ref{reg:section}.  The coordinate-change rule \eqref{setup:surface-equivalence}
makes this measure intrinsic to the quantum surface
\cite{DS11,HRV18}.  We take $0<\gamma<2$ and later restrict to sufficiently
small $\gamma$ for each fixed $W$.

\subsection{Wick ordering and the strip regularization}
\label{setup:wick-section}

We use the Wick-ordering convention of \cite{LRV22}.  Let $X$ be a centered
Gaussian distribution and let $X_\varepsilon$ be the regularization fixed
above.  For every real $\alpha$, set
\begin{equation}\label{setup:Wick}
 :e^{\alpha X_\varepsilon(z)}:
 :=\exp\left(\alpha X_\varepsilon(z)
             -\frac{\alpha^2}{2}\E[X_\varepsilon(z)^2]\right).
\end{equation}
The colons therefore indicate that the diagonal variance has been
subtracted.  They always refer to the centered Gaussian field appearing
inside the colons.  In particular, the radial process of a quantum disk
is not Wick ordered.

The Wick powers are defined by
\begin{equation}\label{setup:Wick-powers}
 :X_\varepsilon(z)^n:
 :=\left.\frac{\partial^n}{\partial\alpha^n}
                    :e^{\alpha X_\varepsilon(z)}:\right|_{\alpha=0}.
\end{equation}
For example,
\[
 :X_\varepsilon^2:
 =X_\varepsilon^2-\E[X_\varepsilon^2],\qquad
 :X_\varepsilon^3:
 =X_\varepsilon^3-3\E[X_\varepsilon^2]X_\varepsilon.
\]
Thus the identity
\begin{equation}\label{setup:Wick-series}
 :e^{\alpha X_\varepsilon}:
 =\sum_{n=0}^\infty\frac{\alpha^n}{n!}:X_\varepsilon^n:
\end{equation}
holds for every fixed regularization scale.  For the limiting GFF, this
series is used only after pairing with a test function or a suitable
reference measure.

If $\rho$ is such a measure, our notation
\begin{equation}\label{setup:Wick-integral}
 \int:e^{\alpha X(z)}:\,\rho(dz)
 :=\lim_{\varepsilon\downarrow0}
             \int:e^{\alpha X_\varepsilon(z)}:\,\rho(dz)
\end{equation}
means convergence in the GMC sense whenever the limit exists.  Wick
ordering affects only the Gaussian field: for a deterministic function
$f$, or conditionally on an independent random function $f$,
\begin{equation}\label{setup:Wick-shift}
 :e^{\alpha(X+f)}:=e^{\alpha f}:e^{\alpha X}:.
\end{equation}

Full variance subtraction differs from subtraction of only the
logarithmic divergence. Indeed, at every fixed cutoff,
\[
 \varepsilon^{\gamma^2/2}e^{\gamma(X_\varepsilon+f)}
 =e^{\gamma f}:e^{\gamma X_\varepsilon}:
   \exp\!\left\{\frac{\gamma^2}{2}
        \bigl(\Var X_\varepsilon-\log\varepsilon^{-1}\bigr)\right\}.
\]
Thus the finite part of the variance determines the deterministic
factor multiplying the Wick measure. Here $f$ is a continuous shift,
or an independent continuous random shift on which we condition.
This correction also appears in \cite[Proposition 1.2]{DS11}.

To specify the finite-variance correction, define $X^\perp$ to be the
projection of a Neumann
GFF on $\Sstrip$ onto the subspace of distributions having zero mean on
every vertical segment.  Equivalently, its covariance kernel is
\begin{equation}\label{setup:covariance}
 G_\perp((t,\theta),(s,\theta'))
 =2\sum_{n\ge1}\frac{e^{-n|t-s|}}n
                         \cos(n\theta)\cos(n\theta').
\end{equation}
For later regularizations, take independent stationary Gaussian processes
$U_n$ with continuous paths and covariance $e^{-n|t-s|}$, and set
\begin{equation}\label{setup:modes}
 X_N^\perp(t,\theta)
 =\sum_{n=1}^N\sqrt{\frac2n}U_n(t)\cos(n\theta).
\end{equation}
Their distributional limit is $X^\perp$.  Every cutoff has zero angular
mean, and both half-strips are restrictions of this same Gaussian field.
We denote the covariance of $X_N^\perp$ by
$G_{\perp,N}(z,w)=\E[X_N^\perp(z)X_N^\perp(w)]$; it is the sum in
\eqref{setup:covariance} restricted to $1\le n\le N$. The covariance
kernel of the unregularized field is interpreted off the diagonal and
as a distribution when paired with smooth test functions.

For this field, summing
the Fourier series gives, for $z=t+i\theta$ and $w=s+i\theta'$,
\begin{align}
 G_\perp(z,w)
 ={}&-\log\left|1-e^{-|t-s|+i(\theta-\theta')}\right|\notag\\
 &-\log\left|1-e^{-|t-s|+i(\theta+\theta')}\right|.
 \label{setup:kernel-sum}
\end{align}
Consequently, at a fixed interior point,
\[
 G_\perp(z,w)
 =-\log|z-w|-\log(2\sin\theta)+o(1)\qquad(w\to z).
\]
The finite part of the circle-average variance is therefore
$-\log(2\sin\theta)$. The preceding cutoff identity gives the
factor $(2\sin\theta)^{-\gamma^2/2}$.
For any continuous radial function $R$, deterministic or independent of
$X^\perp$, the exact quantum-area measure is
\begin{equation}\label{setup:area}
 \mu_{R+X^\perp}^\gamma(dt\,d\theta)
 =e^{\gamma R(t)}:e^{\gamma X^\perp(t,\theta)}:
                     (2\sin\theta)^{-\gamma^2/2}dt\,d\theta.
\end{equation}
The colons retain the full Gaussian Wick convention throughout.  The
reference measure in \eqref{setup:area} is explicitly
$(2\sin\theta)^{-\gamma^2/2}dt\,d\theta$; it is not Lebesgue measure.
We use Euclidean circle averages, or another regularization with the
deterministic correction that gives the same area.  The common coupling
and comparison of regularizations are proved in Subsection~\ref{reg:section}.

\subsection{The maximum-embedded quantum disk of weight \texorpdfstring{$W$}{W}}
\label{setup:disk-section}

We now define the infinite measure $\mathcal M^{\mathrm{disk}}_2(W)$ on
two-pointed quantum disks of weight $W$.  We use the strip
$\Sstrip=\R\times(0,\pi)$, with marked points $-\infty,+\infty$.
For the values of $\gamma$ considered below, $W/\gamma-\gamma/2>0$.
Let $R^+$ and $R^-$ be independent Brownian motions with variance
parameter two and drift $-W/\gamma+\gamma/2$, started from zero and
conditioned to stay nonpositive on $(0,\infty)$.  Here and below this
conditioning means the standard Doob-transform entrance law obtained by
starting at $-\varepsilon$, conditioning never to hit zero, and sending
$\varepsilon\downarrow0$.  Define the two-sided process
\[
 R_\gamma(t)=
 \begin{cases}
 R^+(t),&t\ge0,\\
 R^-(-t),&t<0.
 \end{cases}
\]

In the disk construction, take the lateral field $X^\perp$ just defined
independently of the radial process $R_\gamma$.

Set $\widehat h=R_\gamma+X^\perp$, and independently equip an additive
constant $c$ with the infinite measure
\begin{equation}\label{setup:zero-mode}
 \frac\gamma2
       e^{(\gamma/2-W/\gamma)c}\,dc,
\end{equation}
and set $h=\widehat h+c$.  The pushforward of this product measure under
$(R_\gamma,X^\perp,c)\mapsto(\Sstrip,h,-\infty,+\infty)$, modulo
\eqref{setup:surface-equivalence}, is by definition
$\mathcal M^{\mathrm{disk}}_2(W)$.  This representative is its maximum
embedding: the vertical-average process is pinned at its maximum at
$t=0$, so the representative is not horizontally translation invariant.
This construction agrees with the standard weight-$W$ quantum disk
normalization of \cite{AHS23,DMS21}.

\subsection{Unit-area disintegration and the main theorem}

Because $\mathcal M^{\mathrm{disk}}_2(W)$ is an infinite measure and the
event of having exactly unit area has zero mass, the unit-area disk must
be defined by disintegration.  To use the same mass notation before and
after a deterministic shift, write
\begin{equation}\label{intro:mass}
 \begin{split}
 Z_\gamma(f)
 &=\int_{\Sstrip}e^{f(t,\theta)}(2\sin\theta)^{-\gamma^2/2}
                              :e^{\gamma X^\perp(t,\theta)}:\,dt\,d\theta,\\
 Z_0(f)&=\int_{\Sstrip}e^{f(t,\theta)}\,dt\,d\theta.
 \end{split}
\end{equation}
Here $f$ is deterministic or independent of $X^\perp$, and the notation
is used whenever the indicated integrals exist.
$Z_{\gamma,N}(f)$ denotes the same mass with $X^\perp$ replaced by
$X_N^\perp$ and its own variance subtracted.  Without an argument, put
\[
 Z_\gamma:=\mu_{\widehat h}^\gamma(\Sstrip)
                         =Z_\gamma(\gamma R_\gamma).
\]
The unit-area law is the probability measure under which the
zero-mode-free field is biased by its total area to the power
$W/\gamma^2-1/2$, and then shifted to make that area one. Equivalently,
for a bounded measurable functional $F$, its expectation is defined by
\begin{equation}\label{setup:disintegration}
 \E^{(1)}[F(\gamma h)]
 =\frac{\E\left[F(\gamma\widehat h-\log Z_\gamma)
                         Z_\gamma^{W/\gamma^2-1/2}\right]}
        {\E\left[Z_\gamma^{W/\gamma^2-1/2}\right]}.
\end{equation}
The estimates below show that the denominator is finite and strictly
positive in the parameter range of Theorem~\ref{setup:main} for all
sufficiently small $\gamma$.  Thus, under the unit-area law, the
zero-mode-free field $\widehat h$ is tilted by
$Z_\gamma^{W/\gamma^2-1/2}$, and its unit-area representative is
\begin{equation}\label{setup:unit-representative}
 h_\gamma^{(1)}=\widehat h-\frac1\gamma\log Z_\gamma.
\end{equation}
Indeed, adding $c$ multiplies the area by $e^{\gamma c}$.  With
$a=e^{\gamma c}Z_\gamma$, one has
\[
 c=\frac1\gamma\log\frac a{Z_\gamma},\qquad
 dc=\frac{da}{\gamma a},
\]
and substitution in \eqref{setup:zero-mode} gives
\begin{equation}\label{setup:area-disintegration}
 \frac\gamma2 e^{(\gamma/2-W/\gamma)c}\,dc
 =\frac12a^{-W/\gamma^2-1/2}
                        Z_\gamma^{W/\gamma^2-1/2}\,da.
\end{equation}
Disintegrating this expression at $a=1$ and normalizing proves
\eqref{setup:disintegration} and shows directly that
$\mu_{h_\gamma^{(1)}}^\gamma(\Sstrip)=1$.  Tonelli's theorem justifies
the change of variables first for nonnegative functionals, so it does
not presuppose the moment bound.  Set
\[
 \Phi_\gamma=\gamma h_\gamma^{(1)}.
\]
All chaos measures are constructed from the same Gaussian field;
Subsection~\ref{reg:section} identifies its circle, smooth, and Fourier
regularizations on one probability space.

For precision, a distribution belongs to
$H^{-s}_{\mathrm{loc}}(\Sstrip)$ if its product with every smooth
compactly supported interior cutoff belongs to $H^{-s}(\R^2)$ after
extension by zero. The seminorms of these products define its topology.
The Fourier definition of the global squared norm is
$\int_{\R^2}(1+|\xi|^2)^{-s}|\widehat f(\xi)|^2d\xi$,
with $\widehat f$ the unitary Fourier transform.  In the Gaussian
decomposition we also use the cosine Sobolev space $H_N^{-s}(0,\pi)$,
whose squared norm is
$\sum_{n\ge0}(1+n^2)^{-s}|\langle f,e_n\rangle|^2$, where
$e_0=\pi^{-1/2}$ and $e_n=(2/\pi)^{1/2}\cos(n\theta)$ for $n\ge1$.
The subscript records the Neumann cosine convention.  Define
\begin{equation}\label{setup:spaces}
 \widehat E_s=C_{\mathrm{loc}}(\R;H_N^{-s}(0,\pi)),\qquad
 E_s=\{v\in\widehat E_s:\langle v(t),e_0\rangle=0
                                      \text{ for every }t\}.
\end{equation}
Their restriction to the open strip is continuous into
$H^{-s}_{\mathrm{loc}}(\Sstrip)$, as verified in the proof.  The
deterministic comparison field is
\begin{equation}\label{setup:classical}
 \phi_W(t+i\theta)
 =\log\frac W{4\pi}-2\log\cosh\frac{Wt}{2}.
\end{equation}

For the measure statement, use the compact space
$[-\infty,+\infty]\times[0,\pi]$ with its product topology,
and extend all area measures there by zero outside $\Sstrip$.

\begin{theorem}[Semiclassical unit-area disk]\label{setup:main}
Fix $0<W<2$.  For every $s>0$,
\[
 \Phi_\gamma\longrightarrow\phi_W
 \quad\text{in probability in }H^{-s}_{\mathrm{loc}}(\Sstrip)
 \qquad(\gamma\downarrow0).
\]
Equivalently, for every
$F\in C_b(H^{-s}_{\mathrm{loc}}(\Sstrip))$,
\[
 \lim_{\gamma\downarrow0}\E^{(1)}[F(\Phi_\gamma)]=F(\phi_W).
\]
Moreover, in probability,
\begin{equation}\label{setup:measure}
 \mu_{h_\gamma^{(1)}}^\gamma
                  \Longrightarrow e^{\phi_W(z)}\,d^2z.
\end{equation}
The measure convergence also holds on the stated compactification;
the limiting area measure gives no mass to its boundary.
\end{theorem}

Convergence in probability here means concentration about the deterministic
field and measure; no coupling between different values of $\gamma$ is
required.  The observable statement is proved directly from the normalized
expectation, and the area convergence is proved jointly with the field.

\Needspace{10\baselineskip}
\begin{lemma}[Geometry of the limiting metric]\label{setup:geometry}
The deterministic metric $e^{\phi_W}|dz|^2$ has area one and satisfies
\[
 -\Delta\phi_W=2\pi W e^{\phi_W},\qquad
 \partial_n\phi_W=0\quad\text{on }\partial\Sstrip.
\]
Its Gaussian curvature is $\pi W$, its horizontal boundary is geodesic,
and each marked boundary point has opening angle $\pi W/2$.
\end{lemma}

\begin{proof}
The identity $\int_\R\sech^2(Wt/2)\,dt=4/W$ gives area one.
Differentiation gives
$\Delta\phi_W=-(W^2/2)\sech^2(Wt/2)=-2\pi W e^{\phi_W}$,
while $\phi_W$ is independent of $\theta$.
The curvature formula $-e^{-\phi_W}\Delta\phi_W/2$ gives $\pi W$;
the zero normal derivative along the straight horizontal boundary gives
geodesic boundary.  Near the negative end use $\zeta=e^{t+i\theta}$.
Then
\[
 e^{\phi_W}|dz|^2
 \sim\frac W\pi|\zeta|^{W-2}|d\zeta|^2.
\]
The distance to the end is asymptotic to
$\sqrt{W/\pi}(2/W)|\zeta|^{W/2}$, and the semicircle length is
asymptotic to $\pi\sqrt{W/\pi}|\zeta|^{W/2}$.  Their ratio is
$\pi W/2$.  The coordinate $\zeta=e^{-(t+i\theta)}$ gives the same
calculation at the positive end.
\end{proof}

The disintegration \eqref{setup:disintegration} is the source of the
large exponent $W/\gamma^2$.  No translation invariance of the radial
field is used.

\section{A joint large-deviation principle for the fields and area}
\label{decomp:section}\label{ldp:section}

The goal of this section is to prove one joint LDP for the radial
Gaussian input, the lateral field, and the total area, allowing an
arbitrary fixed finite-energy radial shift. This single result will
cover both the original law and the mass appearing after Girsanov.
Subsection~\ref{ldp:coordinates} represents the radial field as a
continuous reflection of Brownian paths and specifies the allowed
shifts; Subsection~\ref{ldp:gaussian-section} gives the Gaussian LDP
and its energy. Subsection~\ref{ldp:deterministic} controls deterministic
exponentials on bounded energy sets. Subsection~\ref{mom:section}
proves the high-moment estimates that make Fourier and spatial cutoff
errors negligible at speed $\gamma^{-2}$, and
Subsection~\ref{reg:section} identifies these cutoff masses with the
canonical area. Subsection~\ref{ldp:joint-section} then adjoins the area
by continuous approximation, with the order of all cutoffs specified:
the joint rate is the Gaussian energy on the graph of the deterministic
area, and infinity elsewhere. Finally,
Subsection~\ref{ldp:radial-section} expresses this same result in the
original radial-field coordinates. This establishes the probability
LDP needed for the unit-area expectation; the unbounded area reward
and its optimizer are treated in Section~\ref{rad:section}.

\subsection{Gaussian coordinates and radial reflection}
\label{ldp:coordinates}

Our goal here is to express the radial field as a deterministic map of
an independent Gaussian input, so that the LDP can start from a Gaussian
law. We first give the exact Brownian representation, then choose the
path and energy spaces and record the continuity and tail bounds for
the reflection map. Allowing a fixed finite-energy displacement in that
map will make the later Girsanov calculation an application of the same
LDP.

The Neumann GFF Dirichlet space consists of functions of finite
Dirichlet energy, modulo additive constants, with squared norm
$(2\pi)^{-1}\int_{\Sstrip}|\nabla f|^2d^2z$. It decomposes as
$\mathcal H_1(\Sstrip)\oplus\mathcal H_2(\Sstrip)$, where
$\mathcal H_1$ consists of functions with zero vertical mean and
$\mathcal H_2$ consists of functions constant on vertical segments.
The disk construction uses the conditioned radial law in the second
component, giving
\[
 \widehat h(t+i\theta)=R_\gamma(t)+X^\perp(t+i\theta).
\]
The two fields are independent. The process $R_\gamma$ is pinned by
$R_\gamma(0)=0$. On $[0,\infty)$ it is a Brownian motion of variance
parameter two and drift $-W/\gamma+\gamma/2$, conditioned to stay
nonpositive; the process on the negative half-line is an independent
time reversal with the analogous drift. The field $X^\perp$ is the
$\mathcal H_1$-projection of the Neumann GFF.
In particular, for every cutoff $N$ and every $t,s,\vartheta$,
\[
 \int_0^\pi X_N^\perp(t,\theta)d\theta=0,
 \qquad
 \int_0^\pi G_{\perp,N}((t,\theta),(s,\vartheta))d\theta=0.
\]
Each identity follows by integrating $\cos(n\theta)$, $n\ge1$.
The covariance identity passes to the full kernel by its integrable
logarithmic bound in Lemma~\ref{mom:covariance-bound}.
This centering remains part of the Gaussian law throughout the proof.

\subsubsection{The Brownian representation of the radial field}

We now identify the radial law as a reflection of Brownian motion.
We retain $\widehat h=R_\gamma+X^\perp$, with $X^\perp$ denoting
only the lateral field.  Let $B$ be the continuous two-sided process
with $B(0)=0$ whose restrictions $t\mapsto B(t)$ and
$t\mapsto B(-t)$, $t\geq0$, are independent variance-two Brownian
motions: each increment over an interval of length $t$ has variance
$2t$. It is independent of $X^\perp$.  The following form of the
Pitman transform is the entry point for the radial analysis.

\begin{lemma}[Drifted Pitman representation]
\label{setup:Pitman}
For $0<\gamma^2<2W$, the process $R_\gamma$ on either half-line
has the law of
\[
 B(t)-\left(\frac W\gamma-\frac\gamma2\right)t
 -2\sup_{0\leq s\leq t}
       \left\{B(s)-\left(\frac W\gamma-\frac\gamma2\right)s\right\},
 \qquad t\geq0.
\]
The two half-lines in this representation are independent.
\end{lemma}
\begin{proof}
On one half-line put $M_t=\sup_{s\leq t}B(s)$ and
$Y_t=2M_t-B(t)$.  Pitman's theorem \cite{Pitman75} says that $Y$
is the variance-two three-dimensional Bessel process, with generator
$\partial_y^2+(2/y)\partial_y$ and entrance point zero.  Brownian
recurrence gives the pathwise identity
\[
 M_t=\inf_{s\geq t}Y_s.
\]
Indeed, $Y_s\geq M_s\geq M_t$ for $s\geq t$, and the first future
time at which $B$ hits $M_t$ gives equality.  The probability that
this Bessel process, started from $y$, hits $x\in(0,y)$ is $x/y$:
solve $g''+(2/y)g'=0$ with value one at $x$ and zero at infinity.
The Markov property therefore shows that, conditional on the path
of $Y$ up to time $t$, $M_t$ is uniform on $[0,Y_t]$.  Accordingly
$B(t)=2M_t-Y_t$ is conditionally uniform on $[-Y_t,Y_t]$.

Giving $B$ drift $-W/\gamma+\gamma/2$ up to time $t$ multiplies its
law by
\[
 \exp\!\left[-\frac12\left(\frac W\gamma-\frac\gamma2\right)B(t)
       -\frac t4\left(\frac W\gamma-\frac\gamma2\right)^2\right].
\]
Conditional expectation given the path of $Y$ up to time $t$ is
\[
 e^{-\frac t4(W/\gamma-\gamma/2)^2}
 \frac{\sinh((W/\gamma-\gamma/2)Y_t/2)}
      {(W/\gamma-\gamma/2)Y_t/2}.
\]
The ratio has value one at zero and is an eigenfunction of the
Bessel generator with eigenvalue $(W/\gamma-\gamma/2)^2/4$.
Its Doob transform has generator
\[
 \partial_y^2+
 \left(\frac W\gamma-\frac\gamma2\right)
 \coth\!\left[\frac12\left(\frac W\gamma-\frac\gamma2\right)y\right]
 \partial_y.
\]
Negating this process gives the transform in the statement, with
its entrance law at zero.  On the other hand, variance-two
Brownian motion with drift $-W/\gamma+\gamma/2$, started from
$x<0$, avoids zero with probability
$1-e^{(W/\gamma-\gamma/2)x}$.  Its survival Doob transform has drift
\[
 -\left(\frac W\gamma-\frac\gamma2\right)
 -\frac{2(W/\gamma-\gamma/2)e^{(W/\gamma-\gamma/2)x}}
              {1-e^{(W/\gamma-\gamma/2)x}}
 =-\left(\frac W\gamma-\frac\gamma2\right)
       \coth\!\left[-\frac12\left(\frac W\gamma-\frac\gamma2\right)x\right].
\]
This is the same diffusion generator.  The entrance law is the
limit as $x\uparrow0$, proving the stated conditioning convention.
Applying the construction to independent Brownian halves proves
the last assertion.
\end{proof}

Consequently, the total mass before unit-area normalization has the
representation
\begin{equation}
\begin{aligned}
 Z_\gamma(\gamma R_\gamma)\ \overset{\mathrm{law}}=
 \sum_{\sigma\in\{-1,1\}}\int_0^\infty\int_0^\pi
 &(2\sin\theta)^{-\gamma^2/2}:e^{\gamma X^\perp(\sigma t,\theta)}:\\[-2pt]
 &\times\exp\!\left\{\begin{aligned}
 &\gamma B(\sigma t)-(W-\gamma^2/2)t\\[-2pt]
 &-2\sup_{0\leq s\leq t}
         \{\gamma B(\sigma s)-(W-\gamma^2/2)s\}
 \end{aligned}\right\}\dd\theta\,\dd t.
\end{aligned}
\label{rad:Pitman-mass}
\end{equation}
Only the lateral field is Wick ordered in this formula; the radial
Brownian motions carry ordinary exponential weights.  The two
restrictions of $X^\perp$ belong to the same field and are not
independent.

For continuity and tail estimates, define the radial Banach space
\begin{equation}
 \mathcal B=\left\{u\in C(\R):u(0)=0,\quad
 \lim_{|t|\to\infty}\frac{u(t)}{(1+|t|)^{3/4}}=0\right\},
 \qquad
 \|u\|_*:=\sup_{t\in\R}\frac{|u(t)|}{(1+|t|)^{3/4}}.
 \label{setup:path-space}
\end{equation}
Every finite-energy path belongs to this space by
$|u(t)|\leq\sqrt{|t|}\,\|u'\|_{L^2(\R)}$.
It contains the Brownian paths almost surely; the corresponding
Gaussian realization and exponential integrability are verified
in Subsection~\ref{ldp:gaussian-section}.

For $u\in\mathcal B$, $t\ge0$, and $\sigma\in\{-1,1\}$, define the
unshifted reflection map by
\begin{equation}
 \begin{split}
 d_\gamma(u;\sigma t)
 &=u(\sigma t)-(W-\gamma^2/2)t\\
 &\quad-2\sup_{0\le r\le t}
             \{u(\sigma r)-(W-\gamma^2/2)r\}.
 \end{split}
 \label{ldp:unshifted-map}
\end{equation}
Lemma~\ref{setup:Pitman} identifies $d_\gamma(\gamma B)$ with
$\gamma R_\gamma$ in law, jointly with the independent lateral field.
No change of measure has been made.  The supremum contains zero and
the current value, so $d_\gamma(u)\le0$.  For $\gamma^2\le W$,
\begin{equation}
 e^{d_\gamma(u;t)}
 \le\min\{1,e^{-W|t|/2+\|u\|_*(1+|t|)^{3/4}}\}.
 \label{ldp:unshifted-envelope}
\end{equation}
The running maximum is Lipschitz on bounded time intervals, and
\[
 |d_\gamma(u;t)-d_0(u;t)|\le\gamma^2|t|/2.
\]
Indeed, adding $\gamma^2r/2$ to the path inside the maximum changes
that maximum by a number between zero and $\gamma^2|t|/2$.
Consequently these maps are continuous for local uniform convergence,
and $d_\gamma\to d_0$ uniformly on each bounded time interval,
uniformly in the input $u$.

\subsubsection{The lateral energy space}

We identify the lateral energy space so that its norm determines the
Gaussian cost and the deterministic area estimates below. The lateral
Cameron--Martin space is
\[
 \mathcal H_\perp=
 \left\{v\in H^1(\Sstrip):
       \int_0^\pi v(t,\theta)\,d\theta=0
       \text{ for almost every }t\right\}.
\]
Its inner product and norm are
\begin{equation}\label{gir:Hilbert}
 \langle v,h\rangle_{\mathcal H_\perp}
 =\frac1{2\pi}\int_{\Sstrip}\nabla v\cdot\nabla h,
 \qquad
 \|v\|_{\mathcal H_\perp}^2
 =\frac1{2\pi}\int_{\Sstrip}|\nabla v|^2.
\end{equation}
There is no prescribed boundary condition in this form domain.
The zero angular mean gives
$\int|v|^2\le\int|v_\theta|^2$, so the gradient norm is equivalent to
the $H^1$ norm on this space.

\subsubsection{A fixed displacement in the radial map}

We now make the reflection estimates uniform on bounded sets of inputs
for each fixed displacement. This is the only extension needed to use
the same LDP after the radial change of variables in Section~\ref{rad:section}.
Fix a locally absolutely continuous $g$ with $g(0)=0$ and
$g'\in L^2(\R)$. Then $g\in\mathcal B$, and we use
$d_\gamma(u+g)$ while keeping the Gaussian input $u$ centered.
In particular, the cost will be the energy of $u$, not of $u+g$.
For $\|u\|_*\le r$ and $\gamma^2\le W$,
\begin{equation}\label{ldp:tail-envelope}
 e^{d_\gamma(u+g;t)}
 \le\min\{1,e^{-W|t|/2+(r+\|g\|_*)(1+|t|)^{3/4}}\}.
\end{equation}
All limits below hold for each fixed $g$; no uniformity over unbounded
sets of shifts is asserted.

For the moment and regularization estimates, $f_\gamma:\R\to\R$ denotes a deterministic
family of continuous radial exponents, extended to the strip by
$f_\gamma(t,\theta)=f_\gamma(t)$.  We assume $f_\gamma\le0$,
$f_\gamma\to f_0$ locally uniformly as $\gamma\downarrow0$, and that
there is a fixed $r\ge0$ such that, for all sufficiently small $\gamma$,
\begin{equation}
 e^{f_\gamma(t)}
 \le\min\left\{1,
       4\exp\left[-\frac W2|t|+r(1+|t|)^{3/4}\right]\right\}.
 \label{decomp:radial-envelope}
\end{equation}
The same bound holds for $e^{f_0}$ by passage to the limit.
The parameter $r$ only measures the size of this deterministic envelope.
This formulation permits the lateral analysis to be carried out before
choosing a radial path; independent random radial densities will be
handled by conditioning.  The mass notation $Z_\gamma(f_\gamma+v)$
uses $f_\gamma(t)+v(t,\theta)$ as its exponent for a deterministic
shift $v$ on the strip, and
$Z_0(f_0+v)=\int_{\Sstrip}e^{f_0(t)+v(t,\theta)}\,\dd t\,\dd\theta$.

\subsection{The Gaussian LDP}
\label{ldp:gaussian-section}

The goal of this subsection is a good LDP for the centered Gaussian
pair $(\gamma B,\gamma X^\perp)$, where $B$ is the two-sided
variance-two Brownian input and $X^\perp$ is the independent
zero-angular-mean lateral field from Subsection~\ref{ldp:coordinates}.
We also establish compact energy sublevels and exponential tightness. These
properties will permit continuous finite-cutoff contraction and the
removal of the cutoffs.

We recall the meaning of a large-deviation principle (LDP). Let
$Y_\gamma$ be random variables in a metrizable space and let $I$ be
a lower semicontinuous function from that space to $[0,+\infty]$.
Their laws have an LDP at speed $\gamma^{-2}$ with rate $I$ if
\[
 \liminf_{\gamma\downarrow0}\gamma^2\log\P(Y_\gamma\in O)
       \ge-\inf_O I,
 \qquad
 \limsup_{\gamma\downarrow0}\gamma^2\log\P(Y_\gamma\in C)
       \le-\inf_C I
\]
for open $O$ and closed $C$. We use the conventions $\inf\varnothing=+\infty$ and
$\log0=-\infty$. The rate is called good when its sublevel sets
are compact. Exponential tightness means that, for every $L>0$, there
is a compact set $K$ for which the last logarithmic upper bound for
$\P(Y_\gamma\notin K)$ is at most $-L$. All spaces used below are
separable and metrizable.

We use the radial Banach space $\mathcal B$ from
\eqref{setup:path-space} and the cosine Sobolev spaces from
\eqref{setup:spaces}. Recall
\begin{equation}\label{ldp:spaces}
 \widehat E_s=C_{\mathrm{loc}}(\R;H_N^{-s}(0,\pi)),\qquad
 E_s=\{v\in\widehat E_s:\langle v(t),e_0\rangle=0\text{ for all }t\}.
\end{equation}
Their topology is uniform convergence on bounded time intervals in
the indicated angular norm. The lateral Cameron--Martin space is
$\mathcal H_\perp$, defined earlier.

The Gaussian rate at a finite-energy pair $(u,v)$ is
\begin{equation}\label{ldp:gaussian-cost}
 I_G(u,v):=\frac14\int_\R|u'(t)|^2\,dt
 +\frac1{4\pi}\int_{\Sstrip}|\nabla v|^2\,d^2z.
\end{equation}
It is infinite unless $u$ is locally absolutely continuous with
$u'\in L^2(\R)$ and $v\in\mathcal H_\perp$.

\begin{lemma}\label{ldp:gaussian}
The laws of $(\gamma B,\gamma X^\perp)$ satisfy a good large deviation
principle on $\mathcal B\times E_s$, at speed $\gamma^{-2}$,
with the rate specified in \eqref{ldp:gaussian-cost}.
These laws are exponentially tight.
\end{lemma}

\begin{proof}
The Brownian motion belongs to $\mathcal B$ almost surely.  Indeed, the
reflection principle gives, for each $\epsilon>0$,
\[
 \sum_{n\ge1}\P\!\left(\sup_{|t|\le n+1}|B(t)|>
                                  \epsilon n^{3/4}\right)
 \le 8\sum_{n\ge1}e^{-\epsilon^2 n^{3/2}/(4(n+1))}<\infty.
\]
Borel--Cantelli, for rational $\epsilon>0$, proves the required weighted
tail limit.  The resulting law is a centered Radon Gaussian law on the
separable space $\mathcal B$. Its Cameron--Martin squared norm is
$\frac12\int|u'|^2$, because its variance on each half-line is $2t$.
Fernique's theorem states that a centered Gaussian random variable $Y$
in a separable Banach space satisfies
$\E\exp(\epsilon\|Y\|^2)<\infty$ for some $\epsilon>0$.
Applied to the Brownian path as a random element of $\mathcal B$, it gives
\begin{equation}\label{ldp:fernique}
 \E e^{\epsilon\|B\|_*^2}<\infty
 \quad\text{for some }\epsilon>0.
\end{equation}
Here this exponential integrability also has an elementary Brownian
proof. If $\|B\|_*>r$, either $\sup_{|t|\le1}|B(t)|>r$ or, for some
$k\ge0$, $\sup_{|t|\le2^{k+1}}|B(t)|>r2^{3k/4}$.
The reflection principle and the Gaussian tail bound, on the two
half-lines, therefore give for $r\ge1$
\[
 \P(\|B\|_*>r)
 \le8e^{-r^2/4}+8\sum_{k\ge0}e^{-r^2 2^{k/2}/8}
 \le C e^{-c r^2}
\]
for constants $C,c>0$. In the last inequality one can use
$2^{k/2}\ge1+(\sqrt2-1)k$ and sum a geometric series.
Integrating this tail proves \eqref{ldp:fernique} for $0<\epsilon<c$.
In particular, it gives
$\P(\|\gamma B\|_*>r)\le C_\epsilon e^{-\epsilon r^2/\gamma^2}$,
where $C_\epsilon=\E e^{\epsilon\|B\|_*^2}$; this is the estimate
used later to control exceptional radial paths.

The construction of $X^\perp$ from independent angular modes gives, for
$0<|h|\le1$,
\[
 \E\|X^\perp(t+h)-X^\perp(t)\|_{H_N^{-s}}^2
 \lesssim_s\sum_{n\ge1}n^{-1-2s}\min\{1,n|h|\}.
\]
For every $0<r<\min\{2s,1\}$, splitting this sum at $|h|^{-1}$
bounds it by a constant times $|h|^r$. Gaussian higher moments and
the continuity criterion, applied on each bounded time interval,
give an $E_s$-valued version. More explicitly, for an integer $k$
large enough that $kr/2>1$, a centered Hilbert-valued Gaussian variable
satisfies
$\E\|Y\|^k\lesssim_k(\E\|Y\|^2)^{k/2}$.
The increment bound is then $O(|h|^{kr/2})$, so Kolmogorov's
continuity criterion applies. Taking a countable exhaustion of the
time line produces compatible continuous versions. In the normalized
cosine coordinates,
the $n$th covariance is $(\pi/n)e^{-n|t-t'|}$. Applying
$-d^2/dt^2+n^2$ to this kernel gives $2\pi\delta_0$; hence the
Cameron--Martin squared norm of $X^\perp$ is
$\frac1{2\pi}\int|\nabla v|^2$. To apply the Banach-space Gaussian
LDP \cite{Baldi,DZ}, first restrict the lateral field to $[-j,j]$.
The pair $Y_j=(B,X^\perp|_{[-j,j]})$ is a centered Radon Gaussian
random variable in the separable Banach space
$\mathcal B\times C([-j,j];H_N^{-s}(0,\pi))$, with zero angular
mean in its second coordinate. Its small-noise rate is one half
of its squared Cameron--Martin norm. The restricted lateral norm
is the infimum of the global norm over extensions. The
Dawson--G\"artner projective-limit theorem \cite{DZ} therefore
gives an LDP on $\mathcal B\times E_s$ with rate the supremum
of these restricted rates. This supremum equals
\eqref{ldp:gaussian-cost}: if it is finite, choose global
Cameron--Martin extensions on $[-j,j]$ with uniformly bounded
energies. A weakly convergent subsequence in $\mathcal H_\perp$
agrees with the specified field on every bounded interval and
has no larger energy. The converse inequality follows by restricting
any global extension. Independence adds the radial cost, giving
the factors $1/4$ and $1/(4\pi)$.

Exponential tightness in the projective-limit topology can be checked
separately. By the compactly embedded Gaussian Banach space of
\cite[Theorem 1.1 and Section 2]{Baldi}, each $Y_j$ has an almost
surely finite norm $\|\cdot\|_j^\sharp$ whose balls have compact
closure in its ambient Banach space, and
$C_j:=\E\exp(\epsilon_j\|Y_j\|_j^{\sharp\,2})<\infty$ for some
$\epsilon_j>0$. Given $L>0$, choose $r_j$ so that
$\epsilon_jr_j^2\ge L+j+\log C_j$, and let $K_j$ be the closure
of that norm ball. For $0<\gamma\le1$, Markov gives
\[
 \P(\gamma Y_j\notin K_j)
 \le C_j e^{-\epsilon_jr_j^2/\gamma^2}
 \le e^{-(L+j)/\gamma^2}.
\]
The set of pairs whose restriction belongs to $K_j$ for every $j$
is compact in $\mathcal B\times E_s$, by a diagonal subsequence
and compatibility of the restrictions. A union bound bounds its
complement by $\sum_{j\ge1}e^{-(L+j)/\gamma^2}$, whose logarithmic
upper rate is at most $-L$. This proves exponential tightness
without inferring it merely from compactness of rate sublevels.

The rate sublevels are compact in the stated topology. A radial
energy bound gives local equicontinuity and $|u(t)|\lesssim\sqrt{|t|}$;
the latter makes the weighted tails uniformly small. For the
lateral component, zero angular mean gives
$\|v\|_{L^2(\Sstrip)}\le\|v_\theta\|_{L^2(\Sstrip)}$. A bounded
energy set therefore satisfies
\[
 \sup_t\|v(t)\|_{L^2(0,\pi)}^2
 \le2\|v\|_{L^2(\Sstrip)}\|v_t\|_{L^2(\Sstrip)},\qquad
 \|v(t+h)-v(t)\|_{L^2(0,\pi)}
 \le |h|^{1/2}\|v_t\|_{L^2(\Sstrip)}.
\]
The compact embedding $L^2(0,\pi)\hookrightarrow H_N^{-s}(0,\pi)$
and Arzela--Ascoli give compactness in $E_s$. Weak lower
semicontinuity of both energies gives closedness of the sublevels.
\end{proof}

\subsection{Deterministic exponential estimates}\label{ldp:deterministic}

To adjoin the area to the Gaussian LDP, Subsection~\ref{ldp:joint-section}
will approximate the deterministic integral
$\int_{\Sstrip}e^{d_0(u+g;t)+v(t,\theta)}\,dt\,d\theta$
by restricting to $[-T,T]\times(0,\pi)$ and replacing $v$ by its first
$N$ angular modes. We need the resulting errors to vanish uniformly
when $I_G(u,v)$ is bounded. This subsection proves the exponential
integrability of $v$ that makes those approximations possible, using
the strip Onofri inequality and its unit-cell consequence. The sharp
inequality will also identify the optimizer in Subsection~\ref{var:section}.

\Needspace{8\baselineskip}
\subsubsection{The strip Onofri inequality}

We prove the sharp exponential inequality on the strip. Its constants
will determine the optimizer in Section~\ref{var:section}, while its
cell consequence below provides the uniform integrability needed for
the deterministic area map.

Recall the deterministic profile defined in \eqref{intro:classical}:
\[
 \phi_W(t,\theta)=\log\frac W{4\pi}-2\log\cosh\frac{Wt}{2},
 \qquad e^{\phi_W(t,\theta)}=\frac W{4\pi}\sech^2\frac{Wt}{2}.
\]
Its exponential integrates to one with respect to planar Lebesgue
measure. In this section and the next, we write
$d^2z=dt\,d\theta$ for that measure, with $z=t+i\theta$.

\begin{lemma}
\label{var:onofri}
For $W>0$, suppose that $g$ is $H^1$ on every bounded full strip,
has finite Dirichlet energy on the strip, and
$\int_{\Sstrip}g^2e^{\phi_W}\dd^2 z<\infty$.  Then
$\int e^{g+\phi_W}\dd^2 z$ is finite and
\begin{equation}
\label{var:onofri-formula}
 W\log\int_{\Sstrip}e^{g+\phi_W}\dd^2 z
 \le W\int_{\Sstrip}g e^{\phi_W}\dd^2 z
    +\frac1{4\pi}\int_{\Sstrip}|g_t|^2\dd^2 z
    +\frac{W^2}{16\pi}\int_{\Sstrip}|g_\theta|^2\dd^2 z.
\end{equation}
For $g$ independent of $\theta$, equality holds exactly when
\begin{equation}
\label{var:onofri-equality}
 g(t)=c-2\log\{1+\kappa\tanh(Wt/2)\},
 \qquad c\in\R,\quad -1<\kappa<1.
\end{equation}
\end{lemma}

\begin{proof}
The sharp spherical Onofri inequality \cite{Onofri,Mancini}, with
$\dd\omega=\dd A/(4\pi)$ on the unit sphere, is
\begin{equation}
\label{var:sphere}
 \log\int_{\mathbb S^2}e^h\dd\omega
 \le\int_{\mathbb S^2}h\dd\omega
       +\frac1{16\pi}\int_{\mathbb S^2}|\nabla h|^2\dd A.
\end{equation}
In cylinder coordinates its metric is
$\sech^2 s\,(\dd s^2+\dd\theta^2)$.  Apply this inequality to the
even reflection, in the angular variable, of $g(2s/W,\theta)$ from
$0<\theta<\pi$ to the circle of length $2\pi$.
Reflection requires no boundary condition on the normal derivative.
Conformal invariance of Dirichlet energy and the weighted $L^2$
assumption give finite spherical $L^2$ norm and finite Dirichlet energy
away from the poles. To justify extension across a pole, first truncate
the function between $-M$ and $M$. In a conformal coordinate about the
pole, use a radial cutoff that is zero for $r\le\varepsilon^2$, is
$\log(r/\varepsilon^2)/\log(1/\varepsilon)$ for
$\varepsilon^2<r<\varepsilon$, and is one for $r\ge\varepsilon$.
Its Dirichlet energy is $2\pi/\log(1/\varepsilon)$.
Multiplying the truncated function by this cutoff changes its $H^1$
norm by a quantity tending to zero: the extra gradient term is bounded
by $M^2$ times this energy, and the original gradient and $L^2$ terms
on the shrinking disk tend to zero by integrability. Do this at both
poles, and then let $M\to\infty$. The truncations converge in $H^1$
by dominated convergence applied to the function and its gradient.
This proves that the reflected function extends to $H^1(\mathbb S^2)$.

There is no integrability assumption on the exponential in this
argument.  Approximation on the sphere by smooth $H^1$ functions,
followed by an almost-everywhere convergent subsequence and Fatou's
lemma in \eqref{var:sphere}, proves both its integrability and the
inequality at this regularity.  The equality classification of the
sharp inequality also holds for $H^1$ functions.

Under $s=Wt/2$, the reflected spherical probability measure becomes
$\frac{W}{4\pi}\sech^2(Wt/2)\dd t\dd\theta=e^{\phi_W}\dd^2 z$.
The reflected energy is twice the energy on the half cylinder, and
\[
 \int|\partial_s g(2s/W,\theta)|^2\dd s\dd\theta
    =\frac2W\int_{\Sstrip}|g_t|^2\dd^2 z,\qquad
 \int|g_\theta(2s/W,\theta)|^2\dd s\dd\theta
    =\frac W2\int_{\Sstrip}|g_\theta|^2\dd^2 z.
\]
Multiplying \eqref{var:sphere} by $W$ gives
\eqref{var:onofri-formula}.  The spherical equality functions are
$c-2\log(a+b\cdot\omega)$ with $a>|b|$.  Independence of longitude
forces $b$ to lie on the polar axis; its polar coordinate is $\tanh s$.
Absorbing $a$ into $c$ gives \eqref{var:onofri-equality}.
\end{proof}

\subsubsection{Exponential integrability on unit cells}

We deduce a bound for lateral exponentials on every unit cell with a
constant depending only on the total Dirichlet energy. This controls
finite-mode approximation on boxes and, after multiplication by the
radial envelope, the area in the horizontal tails.

\begin{lemma}
\label{var:cell-exponential}
For $v\in\mathcal H_\perp$, $q\ge1$, and $j\in\mathbb Z$,
\begin{equation}
\label{var:cell-exponential-formula}
 \int_{[j,j+1]\times(0,\pi)}e^{qv}\dd^2 z
 \le2\pi\cosh^2(1)
       \exp\left\{\frac{q^2}{8\pi}
                         \int_{\Sstrip}|\nabla v|^2\dd^2 z\right\}.
\end{equation}
\end{lemma}

\begin{proof}
Apply the translated version of Lemma~\ref{var:onofri} with parameter
$2$ to $qv$.  Its source vanishes by angular mean zero, yielding
\[
 \int_{\Sstrip}e^{qv(t,\theta)}
           \frac{\sech^2(t-j)}{2\pi}\dd^2 z
 \le\exp\left\{\frac{q^2}{8\pi}\int_{\Sstrip}|\nabla v|^2\dd^2 z\right\}.
\]
The density is at least $(2\pi\cosh^2(1))^{-1}$ on the stated cell.
\end{proof}

\subsection{GMC moment estimates}
\label{mom:section}

The goal is to control cutoff errors at arbitrarily large exponential
cost, and later to justify the unbounded area reward. We prove bounds
on $L^{q_\gamma}(\mathbb P)$ norms with $q_\gamma=1+(W+2)/(2\gamma^2)$:
these are moments of diverging order, even though the norm takes their
$1/q_\gamma$ power. Covariance and collision estimates give a uniform
bound on each cell; the radial envelope then controls the whole strip
and its tails. Finally, a fixed-order estimate and Gaussian
hypercontractivity make the Fourier cutoff error tend to zero in this
same norm. Markov's inequality converts that convergence into the
exponentially good approximation required by the LDP.

The choice of moment order serves two precise estimates. For an
approximation error $D_\gamma$, Markov's inequality reads
\[
 \gamma^2\log\P(|D_\gamma|>\eta)
 \le\gamma^2 q_\gamma
       \log\frac{\|D_\gamma\|_{L^{q_\gamma}(\mathbb P)}}{\eta}.
\]
Since $\gamma^2q_\gamma\to(W+2)/2>0$, making this probability norm
tend to zero makes the logarithmic upper bound tend to $-\infty$.
For the later area-weighted expectation, put
$p_\gamma=W/\gamma^2-1/2$. A small-probability event $E$ can have a
large contribution to $\E[Z_\gamma^{p_\gamma}\mathbf1_E]$ because
the mass itself is random. H\"older's inequality gives the quantitative
bound
\[
 \E[Z_\gamma^{p_\gamma}\mathbf1_E]
 \le\|Z_\gamma\|_{L^{q_\gamma}(\mathbb P)}^{p_\gamma}
        \P(E)^{1-p_\gamma/q_\gamma}.
\]
Here $q_\gamma>p_\gamma$ and
$1-p_\gamma/q_\gamma\to(2-W)/(W+2)>0$, so a sufficiently negative
probability rate still suppresses this weighted contribution.
Subsection~\ref{ldp:weighted-section} will apply this reasoning to
the full weight, including its Gaussian source. On the other hand,
the collision estimate below permits moment orders with
$\gamma^2q<2$, with a strict margin for the boundary weight.
Choosing $\gamma^2q_\gamma\to(W+2)/2$, between $W$ and $2$,
meets both requirements. We round $q_\gamma$ upward only when using
the integer-moment identity.

For clarity, every probability norm below is written with its measure:
for a real random variable $Y$ and $1\le p<\infty$,
\[
 \|Y\|_{L^p(\mathbb P)}:=\bigl(\E|Y|^p\bigr)^{1/p}.
\]
Thus the norm of a random mass means its moment followed by the
indicated root, not a spatial integral of that mass.

All Gaussian fields in this section use the common mode coupling
\[
 X_N^\perp(t,\theta)=\sum_{n=1}^N\sqrt{2/n}\,U_n(t)\cos(n\theta),
 \qquad \E[U_n(t)U_n(s)]=e^{-n|t-s|},
\]
where the continuous stationary processes $U_n$ are independent.
Their covariance is denoted by $G_{\perp,N}$, and $G_\perp$ is its off-diagonal
limit. For $D=I\times(0,\pi)$ with $I$ a bounded interval and a
deterministic Borel density $0\leq a\leq1$, write
\begin{equation}
 Z_{\gamma,N;D}(\log a)
 :=\int_D a(z)(2\sin\theta)^{-\gamma^2/2}
             :e^{\gamma X_N^\perp(z)}:\,\dd^2 z.
 \label{mom:local-mass}
\end{equation}
Here $\log0=-\infty$, interpreted through the displayed integral.
The notation without $N$ denotes its martingale limit, constructed
below.  Constants implicit in $\lesssim$ are independent of
$\gamma,N$ and $a$; their permitted dependencies are specified in
each statement.

\begin{lemma}[Uniform covariance bound]
\label{mom:covariance-bound}
For distinct points $z,w\in\Sstrip$,
\begin{equation}
 |G_{\perp,N}(z,w)|\leq2\log^+\frac1{|z-w|}+O(1),
 \qquad
 \E[X_M^\perp(z)X_N^\perp(w)]=G_{\perp,\min(M,N)}(z,w),
 \label{mom:covariance-formula}
\end{equation}
where the bounded term is absolute and uniform in the cutoffs and
horizontal translations.  The same bound holds for $G_\perp$.
\end{lemma}
\begin{proof}
The mixed covariance follows directly from independence of the
modes.  To bound a single cutoff, for $|q|\leq1$, $q\ne1$, first
suppose $N|1-q|\leq1$.  Then
\[
 \left|\Re\sum_{n=1}^N\frac{q^n}{n}\right|
       \leq1+\log N\leq1+\log(1/|1-q|).
\]
If $N|1-q|>1$, summation by parts and
$|\sum_{n=N+1}^{m}q^n|\leq2/|1-q|$ give
$|\sum_{n>N}q^n/n|\leq2/((N+1)|1-q|)\leq2$.
The full series has real part $-\log|1-q|$; when $|q|=1$, use
Dirichlet convergence.  Since $|1-q|\leq2$, we have in either case
\[
 \left|\Re\sum_{n=1}^N\frac{q^n}{n}\right|
                     \leq\log^+(1/|1-q|)+O(1).
\]
For $s\geq0$,
\[
 |1-e^{-s+ix}|^2=(1-e^{-s})^2+2e^{-s}(1-\cos x)
 \gtrsim\min\{1,s^2+\operatorname{dist}(x,2\pi\mathbb Z)^2\}.
\]
Use this with $s=|t-t'|$ and $x=\theta-\theta'$ or
$x=\theta+\theta'$.  The corresponding angular distances satisfy
\[
 \operatorname{dist}(\theta-\theta',2\pi\mathbb Z)=|\theta-\theta'|,
 \qquad
 \operatorname{dist}(\theta+\theta',2\pi\mathbb Z)
       =\min(\theta+\theta',2\pi-\theta-\theta')\geq|\theta-\theta'|.
\]
The identity $2\cos(n\theta)\cos(n\theta')
=\cos(n(\theta-\theta'))+\cos(n(\theta+\theta'))$ proves the bound.
The same argument gives the off-diagonal limit
\begin{equation}
 G_\perp(z,w)=-\log|1-e^{-|t-t'|+i(\theta-\theta')}|
         -\log|1-e^{-|t-t'|+i(\theta+\theta')}|.
 \label{mom:full-covariance}
\end{equation}
Convergence is absolute if $t\ne t'$ and follows from Dirichlet's
test otherwise.
\end{proof}

Expanding a Gaussian moment produces one singular factor for each pair
of integration points.  The following deterministic bound integrates all
collisions at once and displays its dependence on the number of replicas.
For a measurable set $D$, $|D|$ denotes planar area, $D^q$ its $q$-fold
product, and $d\mathbf z=d^2z_1\cdots d^2z_q$ the product measure.

\begin{lemma}[Collision integral]
\label{mom:collision-integral}
For a bounded measurable $D\subset\R^2$, an integer $q\geq2$ and
$0<s<4/q$,
\begin{equation}
 \int_{D^q}\prod_{i<j}(1\wedge|z_i-z_j|)^{-s}\,\dd\mathbf z
 \leq |D|\left(|D|+\frac{2\pi}{2-sq/2}\right)^{q-1}.
 \label{mom:collision-bound}
\end{equation}
\end{lemma}
\begin{proof}
Every unordered pair occurs in two of the products below, so
\[
 \prod_{i<j}(1\wedge|z_i-z_j|)^{-s}
 =\prod_{i=1}^q
   \left(\prod_{j\ne i}(1\wedge|z_i-z_j|)^{-sq/2}\right)^{1/q}.
\]
Apply H\"older on this product space with $q$ equal exponents.
The integral of the left side is at most
\[
 \prod_{i=1}^q\left[
 \int_{D^q}\prod_{j\ne i}(1\wedge|z_i-z_j|)^{-sq/2}
                   d\mathbf z\right]^{1/q}.
\]
In its $i$th factor, hold $z_i$ fixed and integrate the other
$q-1$ variables first. There are no pair interactions between
those variables in this factor, so their integrals multiply.  Each inner integral is
at most
$|D|+2\pi\int_0^1\rho^{1-sq/2}\,\dd\rho
=|D|+2\pi/(2-sq/2)$.
Integration in $z_i$ supplies the remaining factor $|D|$.
This treats all simultaneous collisions and all collections of
clusters with an explicit dependence on $q$.
\end{proof}

We apply this collision estimate to the Gaussian replica identity.  The
choice of a moment strictly above $W/\gamma^2-1/2$ leaves room for
H\"older when the expectation is restricted to unlikely Gaussian paths.

\begin{lemma}[Large moments]
\label{mom:high-moment}
Fix $0<W<2$ and a bounded interval $I$.  For all sufficiently small
$\gamma>0$,
\begin{equation}
 \sup_{N\geq1}\sup_{0\leq a\leq1}
 \big\|Z_{\gamma,N;D}(\log a)\big\|_{L^{\,1+(W+2)/(2\gamma^2)}(\mathbb P)}
                     \lesssim_{W,|I|}1.
 \label{mom:high-moment-bound}
\end{equation}
The cutoff masses converge in this norm to
$Z_{\gamma;D}(\log a)$, with the same bound and with mean
$\int_Da(z)(2\sin\theta)^{-\gamma^2/2}\,\dd^2 z$.
\end{lemma}
\begin{proof}
For an integer $q\geq2$, positivity and Tonelli allow the $q$
spatial integrals and expectation to be interchanged.  The variance
of $\sum_iX_N^\perp(z_i)$ is
$\sum_iG_{\perp,N}(z_i,z_i)+2\sum_{i<j}G_{\perp,N}(z_i,z_j)$.
Consequently the subtracted Wick self-variances cancel exactly, and
\begin{equation}
 \E[Z_{\gamma,N;D}(\log a)^q]
 =\int_{D^q}\prod_i a(z_i)(2\sin\theta_i)^{-\gamma^2/2}
       \exp\!\left(\gamma^2\sum_{i<j}G_{\perp,N}(z_i,z_j)\right)
                                             \dd\mathbf z.
 \label{mom:replica-formula}
\end{equation}
The bounded part of the covariance contributes
$\exp(O(\gamma^2q(q-1)))$.  After taking a $q$-th root it is
$\exp(O(\gamma^2q))$, with an absolute implied constant.
To integrate the remaining factors, fix
$1<p<4/(W+2)$.  H\"older with exponents $p$ and $p/(p-1)$,
separates the factors
$\prod_{i<j}(1\wedge|z_i-z_j|)^{-2\gamma^2}$ and
$\prod_i(2\sin\theta_i)^{-\gamma^2/2}$.
For the first factor use Lemma~\ref{mom:collision-integral} with
$s=2p\gamma^2$; the second factor is a product of $q$ identical
one-point integrals. Taking the $q$th root yields
\begin{align}
 \big\|Z_{\gamma,N;D}(\log a)\big\|_{L^q(\mathbb P)}
 &\leq e^{O(\gamma^2q)}|D|^{1/(pq)}
 \left(|D|+\frac{2\pi}{2-p\gamma^2q}\right)^{(q-1)/(pq)}
 \notag\\[-2pt]
 &\quad\times
 \left[|I|\int_0^\pi(2\sin\theta)^{-p\gamma^2/(2(p-1))}
                                  \,\dd\theta\right]^{(p-1)/p},
 \label{mom:explicit-moment}
\end{align}
provided $p\gamma^2q<2$ and $p\gamma^2/(2(p-1))<1$.
The first condition is exactly the collision-integral condition
with exponent $2p\gamma^2$; the second integrates the boundary
weight.  This also proves that the quantities in
\eqref{mom:replica-formula} are finite.

Now take $q=\lceil1+(W+2)/(2\gamma^2)\rceil$.  Then
\[
 \frac{W+2}{2}<\gamma^2q\leq\frac{W+2}{2}+2\gamma^2,
 \qquad
 2-p\gamma^2q\longrightarrow2-p(W+2)/2>0.
\]
For small $\gamma$, the integrand in the one-point boundary
integral is bounded by
$\max\{1,(2\sin\theta)^{-1/2}\}$, an integrable function.
Every factor of \eqref{mom:explicit-moment} is thus bounded
independently of $\gamma,N,a$.  Lyapunov's inequality passes from
this integer $q$ to the asserted real exponent.

Let $\mathcal G_N=\sigma(U_1,\ldots,U_N)$, including their entire
paths.  Independence and mean-one Wick normalization give, for
$M\geq N$,
\[
 \E[Z_{\gamma,M;D}(\log a)\mid\mathcal G_N]
                         =Z_{\gamma,N;D}(\log a).
\]
The nonnegative martingale is bounded in $L^q(\mathbb P)$, with $q>1$ fixed
once $\gamma$ is fixed.  Uniform integrability gives an almost
sure and $L^1(\mathbb P)$ limit with the stated mean; Fatou gives its $L^q(\mathbb P)$
bound.  The martingale then equals the conditional expectation of
this limit, and approximation by bounded cylinder functions proves
convergence in $L^q(\mathbb P)$. More explicitly, if the limit is $Y$ and $Y_j$ is a
bounded $\mathcal G_j$-measurable approximation in $L^q(\mathbb P)$, then for
$N\ge j$ conditional expectation is a contraction and
\[
 \|\E[Y\mid\mathcal G_N]-Y\|_{L^q(\mathbb P)}
 \le 2\|Y-Y_j\|_{L^q(\mathbb P)}.
\]
Such approximations exist because
$\sigma(\bigcup_j\mathcal G_j)$ contains $Y$. Sending $j\to\infty$
proves convergence in $L^q(\mathbb P)$ and therefore in the asserted norm.  The domination in \eqref{mom:explicit-moment} also permits
$N\to\infty$ in \eqref{mom:replica-formula}, replacing $G_{\perp,N}$ by
$G_\perp$, for every fixed $\gamma,q$ satisfying its two strict
integrability conditions.
\end{proof}

The two-dimensional integral in Lemma~\ref{mom:collision-integral}
is essential: with a one-dimensional base, the same argument would
require $sq/2<1$.  Here the radial weight is a bounded density
relative to horizontal Lebesgue measure, so the horizontal
integration remains available.  The estimate is a norm bound; its
raw moment at order $W/\gamma^2-1/2$ may still have exponential size
$\exp(O(\gamma^{-2}))$.  This is why the lateral field remains in
the joint variational problem.

Lemma~\ref{mom:high-moment} gives convergence of the real-valued
random variables $Z_{\gamma,N;D}(\log a)$ for each fixed box $D$ and
test density $a$. We next show that these limits are the integrals of
a single random measure, so they can also be used to study area
observables. We then prove that every nonempty open set has positive
mass. Together with the whole-strip finiteness bound below, positivity
ensures that the total mass can be divided out and its logarithm taken
in the unit-area field.

\begin{lemma}[The measure defined by the cutoff limits]\label{mom:measure}
For each fixed sufficiently small $\gamma>0$, there is a positive
random Radon measure $\mu_\gamma$ on the strip such that, for every
fixed bounded box $D$ and deterministic Borel density $0\le a\le1$,
\[
 \int_D a(z)\,\mu_\gamma(dz)
 = Z_{\gamma;D}(\log a)
 =\lim_{N\to\infty}Z_{\gamma,N;D}(\log a)
 \quad\text{almost surely and in }L^1(\mathbb P).
\]
Here the limit on the right is the real-valued martingale limit from
Lemma~\ref{mom:high-moment}. The expectation measure is
$(2\sin\theta)^{-\gamma^2/2}d^2z$. In particular, for every fixed
Lebesgue-null set $A$, $\mu_\gamma(A)=0$ almost surely.
\end{lemma}
\begin{proof}
On each closed
box $[-T,T]\times[0,\pi]$, take simultaneous limits against a
countable uniformly dense rational vector space of continuous
functions containing the constants.  Signed tests are handled by
their positive and negative parts.  The cutoff total masses
converge, and
$|\int f\,\dd\mu_N|\leq\|f\|_\infty\mu_N(D)$.
The limit therefore extends to a bounded positive functional on
all continuous functions, hence to a finite positive measure by
the Riesz representation theorem.  Its expectation is
$(2\sin\theta)^{-\gamma^2/2}\dd^2 z$, first on continuous tests and
then as an identity of Borel measures.  All fixed box edges have
zero mass almost surely.  On a countable exhaustion the limits
therefore agree on overlaps and define a positive Radon measure
on the strip, with no mass on its two horizontal boundaries.
Continuous approximation in $L^1((2\sin\theta)^{-\gamma^2/2}\dd^2 z)$
identifies each bounded Borel integral with the limit of the
corresponding real-valued cutoff masses.  The common-coupling comparison in Subsection~\ref{reg:section}
identifies this measure with the geometric chaos used in the disk model.

\end{proof}

\begin{lemma}[Strict positivity]\label{mom:positive}
Almost surely the limiting measure assigns positive mass to every
nonempty open set in the strip.
\end{lemma}
\begin{proof}
Fix a compact interior rectangle of positive area. We prove that
its mass is almost surely nonzero. For each positive integer $m$,
delete the first $m$ modes.  The resulting
cutoff masses form a nonnegative martingale whose covariance is
$G_{\perp,N}-G_{\perp,m}$.  On a fixed interior rectangle $G_{\perp,m}$ and the boundary
weight are bounded.  Lemma~\ref{mom:covariance-bound} and the
two-replica formula give a uniform second moment when $\gamma^2<1$,
because $\int_{D^2}(1\wedge|z-w|)^{-2\gamma^2}\dd^2 z\dd^2 w<\infty$.
The remaining-mode mass thus has an almost sure and $L^2(\mathbb P)$ limit.
The full density is exactly this density multiplied by
$\exp(\gamma X_m^\perp(z)-\gamma^2G_{\perp,m}(z,z)/2)$, a continuous function
bounded above and below by positive random constants on the
rectangle.  After integration and passage to the limit, the full
mass is zero if and only if the remaining-mode mass is zero.
For every $m$ the latter event is measurable with respect to
$\sigma(U_n:n>m)$. Its limit inferior as $m\to\infty$ belongs to
the intersection of these sigma-fields: for each fixed $j$, all
terms with $m\ge j$ are measurable with respect to
$\sigma(U_n:n>j)$. The limit inferior agrees almost surely with
the full zero-mass event. Independence of the modes and Kolmogorov's
zero--one law show that its probability is zero or one. Probability
one would force the full mass to have expectation zero, contrary
to its preserved positive mean. Its probability is therefore zero.  A countable basis of
rectangles then gives positive mass to every nonempty open set.

\end{proof}

The radial envelope now turns the local moment estimate into a total-mass
bound.  Its polynomial dependence on $r$ will later allow conditioning
on an independent random radial density.

\begin{proposition}[Whole-strip radial moments]
\label{mom:whole-strip}
Let $f_\gamma:\R\to\R$ be deterministic continuous radial
exponents, interpreted on the strip as $f_\gamma(t,\theta)=f_\gamma(t)$.
Assume that there is a fixed $r\ge0$ such that, for all sufficiently
small $\gamma$,
\[
 e^{f_\gamma(t)}\le
 \min\{1,4e^{-W|t|/2+r(1+|t|)^{3/4}}\},\qquad t\in\R.
\]
This is the envelope \eqref{decomp:radial-envelope}; no convergence
assumption on $f_\gamma$ is needed for this proposition.
The total mass is $Z_\gamma(f_\gamma)=\int_{\Sstrip}e^{f_\gamma(t)}
\mu_\gamma(dz)$, where $\mu_\gamma$ is the measure from
Lemma~\ref{mom:measure}. For all sufficiently small $\gamma$,
\begin{equation}
 \big\|Z_\gamma(f_\gamma)\big\|_{L^{\,1+(W+2)/(2\gamma^2)}(\mathbb P)}
                      \lesssim_W(1+r)^4.
 \label{mom:whole-strip-bound}
\end{equation}
The same bound holds at every finite mode cutoff.  The norm of the mass
restricted to $|t|>T$ is $\lesssim_W e^{-WT/4}$ once $T$ is sufficiently
large in terms of $W,r$.  These bounds are uniform over all radial
exponents satisfying the displayed envelope.
\end{proposition}
\begin{proof}
The assumed envelope reads
\begin{equation}
 e^{f_\gamma(t)}
 \le\min\{1,4e^{-W|t|/2+r(1+|t|)^{3/4}}\}.
 \label{mom:radial-envelope}
\end{equation}
If $|t|\geq\max\{1,(2^{11/4}r/W)^4\}$, then
$r(1+|t|)^{3/4}\leq W|t|/4$.
There are $O_W((1+r)^4)$ unit intervals meeting the complementary
central region, and on them $e^{f_\gamma}\leq1$.
On the remaining intervals the cell suprema are bounded by the
summable geometric sequence from $4e^{-W|t|/4}$.
Their total is therefore $\lesssim_W(1+r)^4$.

Lemma~\ref{mom:high-moment} bounds the norm of the unweighted mass
on every unit cell uniformly, by horizontal stationarity.
Positivity and Minkowski's inequality, first for a finite sum of
cells and then by monotone convergence, prove
\eqref{mom:whole-strip-bound}.  Summing only the cells meeting
$|t|>T$ proves the tail bound; enlarging the threshold for $T$ by one
accounts for the two cells intersecting its endpoints.
The same lateral field is used on both half-lines, and no independence
between those restrictions is required.

Only the envelope, not local convergence of $f_\gamma$, was used in
this proof.  If the radial exponent and its finite envelope parameter
$r$ are random and independent of the lateral field, condition on
both of them.  The resulting conditional norm is bounded by the same
quantity $\lesssim_W(1+r)^4$, pointwise in the conditioning variable.
\end{proof}

Positivity and \eqref{mom:whole-strip-bound} imply
$0<Z_\gamma(f_\gamma)<\infty$ almost surely, so its logarithm in the
normalized field is well-defined.

\begin{lemma}[Fourier cutoff error]
\label{mom:cutoff-error}
Fix a bounded strip box $D$.  There is a deterministic sequence
$\varepsilon_N\to0$ such that, for all sufficiently small
$\gamma$,
\begin{equation}
 \sup_{0\leq a\leq1}
 \big\|Z_{\gamma;D}(\log a)-Z_{\gamma,N;D}(\log a)
       \big\|_{L^{\,1+(W+2)/(2\gamma^2)}(\mathbb P)}\leq\varepsilon_N.
 \label{mom:cutoff-norm}
\end{equation}
Consequently, for every $\varepsilon>0$,
\begin{equation}
 \lim_{N\to\infty}\limsup_{\gamma\downarrow0}\gamma^2\log
 \sup_{0\leq a\leq1}
 \P\left(\left|Z_{\gamma;D}(\log a)
              -Z_{\gamma,N;D}(\log a)\right|>\varepsilon\right)
                         =-\infty.
 \label{mom:cutoff-probability}
\end{equation}
The corresponding signed integrals obey the same estimates for
deterministic densities of absolute value at most one.  The bounds
also hold conditionally for independent random densities.
\end{lemma}
\begin{proof}
A direct expansion at the diverging target moment order would have
an increasing number of signed terms. Instead we first prove convergence
at one fixed even moment order, and then use hypercontractivity to
reach the target order. Choose an even integer $k>4/(2-W)$. This is
equivalent to $k(W+2)/(2(k-1))<2$, so one can choose $p>1$ such that
\begin{equation}
 \frac{pk(W+2)}{2(k-1)}<2.
 \label{mom:fixed-replica-parameters}
\end{equation}
Temporarily use the noise parameter
\[
 \sqrt{\frac{W+2}{2(k-1)}}
\]
and keep the base weight $(2\sin\theta)^{-\gamma^2/2}$ unchanged.
Expand the even $k$-th
power of the difference of two finite-cutoff masses.  In each
replica choose $n_i=M$ or $n_i=N$ according to a subset
$J\subset\{1,\ldots,k\}$, with $n_i=M$ exactly when $i\in J$.
The mixed covariance in \eqref{mom:covariance-formula} gives the
signed kernel
\begin{equation}
 \sum_{J\subset\{1,\ldots,k\}}(-1)^{k-|J|}
 \exp\!\left\{\frac{W+2}{2(k-1)}
                    \sum_{i<j}G_{\perp,\min(n_i,n_j)}(z_i,z_j)\right\}.
 \label{mom:mixed-replica-kernel}
\end{equation}
Its absolute value is bounded, uniformly in $M\geq N$, by a
constant depending only on $W,k$ times
\begin{equation}
 \prod_{i<j}(1\wedge|z_i-z_j|)^{-(W+2)/(k-1)}.
 \label{mom:fixed-replica}
\end{equation}
This bound is integrable after multiplication by the $k$
boundary weights, uniformly for small $\gamma$.  Indeed,
H\"older with $p,p/(p-1)$ and
Lemma~\ref{mom:collision-integral} integrate its $p$-th power by
\eqref{mom:fixed-replica-parameters}.  Fix a sufficiently small
$\varepsilon>0$ with $p\varepsilon^2/(2(p-1))<1$; for
$0<\gamma\leq\varepsilon$ the remaining weight is dominated by
$\prod_i\max\{1,(2\sin\theta_i)^{-\varepsilon^2/2}\}$, whose
$p/(p-1)$-th power is integrable.

For a configuration of distinct points, all mixed covariances
converge to $G_\perp(z_i,z_j)$ uniformly over $M\geq N$ as
$N\to\infty$.  Every term of \eqref{mom:mixed-replica-kernel}
then has the same limit, while its coefficients sum to zero.
The supremum over the countably many $M\geq N$ of its absolute
value therefore tends to zero almost everywhere.  Dominated
convergence with the preceding envelope proves
\begin{align}
 \sup_{\substack{0<\gamma\leq\varepsilon\\0\leq a\leq1,\ M\geq N}}
 \Bigg\|\int_D a(z)(2\sin\theta)^{-\gamma^2/2}
 \Big(&:e^{\sqrt{(W+2)/(2(k-1))}\,X_M^\perp(z)}:
 \notag\\[-2pt]
 &-:e^{\sqrt{(W+2)/(2(k-1))}\,X_N^\perp(z)}:\Big)\dd^2 z
       \Bigg\|_{L^k(\mathbb P)}\longrightarrow0.
 \label{mom:fixed-replica-convergence}
\end{align}
The signed expansion is legitimate term by term because each
term has this same integrable bound.  Evenness of $k$ makes the
expanded power the absolute $k$-th power.  No full-chaos estimate
was assumed in this two-cutoff argument.

For $0\le\rho\le1$, let $X^{\perp\prime}$ be an independent copy of
$X^\perp$ and let $\E'$ integrate this copy. For a Gaussian functional
$Y$, the Mehler operator is
$T_\rho Y(X^\perp)=\E'[Y(\rho X^\perp+\sqrt{1-\rho^2}X^{\perp\prime})]$.  Gaussian hypercontractivity \cite{BBG} states
\[
 \|T_\rho Y\|_{L^q(\mathbb P)}\leq\|Y\|_{L^k(\mathbb P)},
 \qquad q\geq k>1,\quad \rho^2\leq(k-1)/(q-1).
\]
It holds on the present Gaussian space by finite-dimensional
approximation.  The action on a Wick exponential can be checked
without a series expansion. If $A$ is a centered Gaussian linear
functional of variance $\sigma^2$ and $A'$ is its independent copy,
then
\[
 \E'\exp\left\{\alpha(\rho A+\sqrt{1-\rho^2}A')
                         -\frac{\alpha^2\sigma^2}{2}\right\}
 =\exp\left\{\alpha\rho A-\frac{\alpha^2\rho^2\sigma^2}{2}\right\}.
\]
Thus $T_\rho:e^{\alpha A}:=:e^{\rho\alpha A}:$.
Integration against the deterministic density commutes with $T_\rho$
by integrability of the finite-cutoff masses. Take
\[
 \rho=\gamma\sqrt{\frac{2(k-1)}{W+2}},
 \qquad q=1+\frac{W+2}{2\gamma^2}.
\]
For small $\gamma$, $q\geq k$ and the hypercontractive condition
is an equality.  Applying it to the difference in
\eqref{mom:fixed-replica-convergence} replaces the noise
parameter by $\gamma$ and preserves the deterministic boundary
weight.  It gives the norm bound in \eqref{mom:cutoff-norm}
first between cutoffs $M,N$ and then, by
Lemma~\ref{mom:high-moment}, with $M\to\infty$.
Finally Markov's inequality gives
\[
 \gamma^2\log\sup_a\P\left(
 |Z_{\gamma;D}(\log a)-Z_{\gamma,N;D}(\log a)|>\varepsilon\right)
 \leq\left(\gamma^2+\frac{W+2}{2}\right)
                          \log(\varepsilon_N/\varepsilon),
\]
once $\varepsilon_N<\varepsilon$, proving
\eqref{mom:cutoff-probability}.  Absolute values of products of
signed densities give the same bounds.  Conditioning proves the
independent-density assertion; the supremum is always outside
the norm or probability, not inside the random variable.
\end{proof}

\subsection{Canonical regularization of the mass}
\label{reg:section}

The normalized expectation \eqref{intro:ratio} involves the total chaos
mass. To approximate masses with independent radial densities by continuous
Gaussian functionals, we must identify their canonical chaos with the
Fourier-cutoff chaos on the same probability space.  The comparison must
also hold at moment orders proportional to $\gamma^{-2}$, since the mass
is raised to the power $W/\gamma^2-1/2$.  The analytic comparison below
uses the covariance and moment estimates proved in
Lemma~\ref{mom:covariance-bound}, the fixed-replica bound
\eqref{mom:fixed-replica}, and Lemma~\ref{mom:cutoff-error}.

Fix a bounded strip box $D$, possibly adjoining both horizontal boundaries.
Choose a nonnegative radial smooth probability density $\rho$, supported
in the unit disk, and convolve the even, $2\pi$-periodic extension of $X^\perp$
with $\varepsilon^{-2}\rho(\cdot/\varepsilon)$, $0<\varepsilon<1/4$.
Write the resulting field as $X^\perp_\varepsilon$.  Its kernel on the strip is
\[
 K_\varepsilon((t,\theta),(s,\phi))
 =\sum_{m\in\mathbb Z}\varepsilon^{-2}\left[
 \rho\!\left(\frac{t-s}{\varepsilon},
             \frac{\theta-\phi-2\pi m}{\varepsilon}\right)
 +\rho\!\left(\frac{t-s}{\varepsilon},
              \frac{\theta+\phi-2\pi m}{\varepsilon}\right)\right].
\]
Folding the angular integral shows that this is a nonnegative probability
kernel.  It satisfies
$K_\varepsilon(z,x)\lesssim_\rho
\varepsilon^{-2}\ind_{\{|x-z|\le\varepsilon\}}$ and preserves zero
angular mean.  More explicitly, $X^\perp_\varepsilon$ is the $L^2(\mathbb P)$ limit of
$\int K_\varepsilon(z,x)X^\perp_N(x)\,\dd x$.  Independence of the modes makes
the partial variance sums increasing, while the integrated logarithmic
covariance bound makes them finite.  This construction couples all the
regularizations used below.  Colons always subtract the full variance of
the particular regularized field. In the next lemma, $U$ and $V$
each denote either a Fourier cutoff $X_N^\perp$ or a smooth
regularization $X_\varepsilon^\perp$, with arbitrary choices of their
two cutoff parameters.

\begin{lemma}
\label{reg:mixed-covariance}
All Fourier/Fourier, Fourier/smooth, and smooth/smooth mixed covariances
satisfy
\begin{equation}
 \Cov(U(z),V(w))\le2\log^+\frac1{|z-w|}+O(1),
 \label{reg:covariance-bound}
\end{equation}
uniformly in their cutoffs.  For distinct interior points they converge
jointly to $G_\perp(z,w)$ when both regularizations are removed.
\end{lemma}

\begin{proof}
The common mode coupling gives
\begin{align*}
 \Cov(X^\perp_N(z),X^\perp_\varepsilon(w))
 &=\int K_\varepsilon(w,y)G_{\perp,N}(z,y)\,\dd y,\\
 \Cov(X^\perp_\varepsilon(z),X^\perp_{\varepsilon'}(w))
 &=\iint K_\varepsilon(z,x)K_{\varepsilon'}(w,y)G_\perp(x,y)\,\dd x\,\dd y.
\end{align*}
The passage from finite modes is justified by logarithmic domination.
For any such probability kernel,
\[
 \int K_\varepsilon(w,x)\log^+\frac1{|z-x|}\,\dd x
 \le\log^+\frac1{|z-w|}+O_\rho(1).
\]
If $|z-w|\ge2\varepsilon$, use $|z-x|\ge|z-w|/2$.
Otherwise scaling bounds the integral by
$\log(1/\varepsilon)+O_\rho(1)$, using
$\int_{|u|<1}\log(1/|u|)\,\dd u=\pi/2$.
Iterating this estimate handles two kernels without increasing the
coefficient of the logarithm.  Lemma~\ref{mom:covariance-bound} therefore
proves \eqref{reg:covariance-bound}.

Off the image diagonals, the Fourier series converge uniformly: summation
by parts bounds the logarithmic-series tail by
$2/((N+1)|1-q|)$ for $|q|\le1$, $q\ne1$.
The shrinking supports and mass-one property of the kernels then give
joint convergence.  No relation between $N$ and either smoothing scale
is required.
\end{proof}

For the comparison, choose an even integer $k>4/(2-W)$ and then $p>1$
such that $pk(W+2)/(2(k-1))<2$.  These choices are fixed while
$\gamma\to0$.  Test densities $a$ below are deterministic and measurable,
with $0\le a\le1$; every supremum over $a$ lies outside the probability norm.

\begin{lemma}
\label{reg:comparison}
As $N\to\infty$ and $\varepsilon\downarrow0$, jointly,
\begin{equation}
 \sup_{0\le a\le1}
 \left\|\int_D a(z)(2\sin\theta)^{-\gamma^2/2}
 \big(:e^{\gamma X^\perp_N(z)}:-:e^{\gamma X^\perp_\varepsilon(z)}:\big)\,\dd^2 z
 \right\|_{L^{1+(W+2)/(2\gamma^2)}(\mathbb P)}\longrightarrow0,
 \label{reg:large-moment-comparison}
\end{equation}
uniformly for all sufficiently small $\gamma>0$.  In particular, smooth
and Fourier regularizations define the same weighted chaos on $D$.
\end{lemma}

\begin{proof}
First replace the noise coefficient $\gamma$ in the two Wick exponentials
by $\sqrt{(W+2)/(2(k-1))}$, leaving the weight
$(2\sin\theta)^{-\gamma^2/2}$ unchanged, and expand the even $k$th moment
of their difference.  There are $2^k$ mixed-replica terms, with sign
$(-1)^{k-|S|}$ when the Fourier field is assigned to the indices in
$S\subset\{1,\ldots,k\}$.  Wick subtraction cancels all self variances.
Each term is the integral of the exponential of the pair covariances
multiplied by $(W+2)/(2(k-1))$.

Lemma~\ref{reg:mixed-covariance} and the fixed-replica estimate
\eqref{mom:fixed-replica} give one integrable absolute majorant,
uniform in $a$, in both cutoffs, and in sufficiently small $\gamma$.
For each off-diagonal configuration all exponential terms have the same
limit, whereas their signs sum to $(1-1)^k=0$.
Dominated convergence proves convergence of this $L^k(\mathbb P)$ norm to zero.

The Ornstein--Uhlenbeck semigroup on the common Gaussian space sends
$:e^{\alpha U}:$ to $:e^{\alpha e^{-\tau}U}:$ for every centered Gaussian
linear functional $U$.  Choose
$e^{-\tau}=\gamma\sqrt{2(k-1)/(W+2)}$.
Then $e^{2\tau}(k-1)=(W+2)/(2\gamma^2)$, so hypercontractivity transfers
the preceding $L^k(\mathbb P)$ estimate to
\eqref{reg:large-moment-comparison}.
The Fourier limit exists by Lemma~\ref{mom:cutoff-error}; the triangle
inequality identifies the smooth limit with it.
\end{proof}

We finally relate this construction to Euclidean circle averages.
Circle averages are coupled through the same modes; near the boundary
we may use the even periodic extension.  Inside any compact subset of
the open strip, small enough circles are ordinary Euclidean circles.

\begin{lemma}
\label{reg:canonical}
The common limit in Lemma~\ref{reg:comparison} is the circle-average
Wick chaos.  Its canonical area density is
$(2\sin\theta)^{-\gamma^2/2}:e^{\gamma X^\perp}:$.
\end{lemma}

\begin{proof}
The circle-potential identity
\[
 \frac1{2\pi}\int_0^{2\pi}-\log|q-\varepsilon e^{iu}|\,\dd u
 =-\log\max\{|q|,\varepsilon\}
\]
follows from the harmonic mean-value property when $|q|>\varepsilon$
and from the logarithmic power series when $|q|<\varepsilon$;
the equality case follows by integrability of the logarithmic singularity.
Applied to the two image logarithms, it proves
\eqref{reg:covariance-bound} for Fourier/circle and circle/circle
covariances, with bounded cylinder corrections.  It also gives existence
of the circle averages in Gaussian $L^2(\mathbb P)$.  Their mixed off-diagonal
limits are $G_\perp$, so Lemma~\ref{reg:comparison} applies unchanged.

For $D\cap\{\min(\theta,\pi-\theta)<\varepsilon\}$, the positive
fixed-replica integral tends to zero by dominated convergence.
The same hypercontractivity argument therefore gives
\begin{equation}
 \sup_U\left\|\int_{D\cap\{\min(\theta,\pi-\theta)<\varepsilon\}}
 (2\sin\theta)^{-\gamma^2/2}:e^{\gamma U(z)}:\,\dd^2 z
 \right\|_{L^{1+(W+2)/(2\gamma^2)}(\mathbb P)}\longrightarrow0,
 \label{reg:collar}
\end{equation}
uniformly for small $\gamma$.  The supremum includes Fourier, smooth,
auxiliary circle regularizations, and their limit.  Thus identification
on compact interior sets extends to the whole box.

At interior points, averaging the local covariance expansion gives
\begin{equation}
 \Var X^\perp_\varepsilon(t,\theta)
 =\log(1/\varepsilon)-\log(2\sin\theta)
 -\iint\rho(x)\rho(y)\log|x-y|\,\dd x\,\dd y+o(1).
 \label{reg:finite-part}
\end{equation}
The error is locally uniform, and the double-integral correction is zero
for circle averages by the same potential identity.
Consequently, logarithmic-only normalization of the smooth field requires
the factor
$\exp\{(\gamma^2/2)\iint\rho(x)\rho(y)\log|x-y|\,\dd x\,\dd y\}$
to agree with canonical area.  Full Wick subtraction requires no such
factor and gives exactly the density stated in the lemma.
The expansion \eqref{reg:finite-part} is used only in the interior;
\eqref{reg:collar} handles the boundary independently.
\end{proof}

These conclusions apply to an independent radial density by conditioning.
To identify the circle-average chaos of the full field, fix $\gamma$
and a compact interior set $D$, and condition on its continuous
radial part $R$. Let $R_\varepsilon(z)$ be the circle average of
the function $z\mapsto R(\Re z)$. Uniform continuity on a slightly
larger compact interval gives
\[
 \delta_\varepsilon:=\sup_{z\in D}
       |R_\varepsilon(z)-R(\Re z)|\longrightarrow0.
\]
For any nonnegative continuous test supported in $D$, its
regularized full-field mass, using
$\varepsilon^{\gamma^2/2}e^{\gamma(R_\varepsilon+X^\perp_\varepsilon)}$,
lies between $e^{-\gamma\delta_\varepsilon}$ and
$e^{\gamma\delta_\varepsilon}$ times the corresponding mass using
$R(\Re z)$ instead of $R_\varepsilon(z)$. The latter converges by
the preceding common-coupling identification. The sandwich proves
\eqref{setup:area} on $D$, and a countable interior exhaustion
identifies the measures on the whole open strip. This argument
is made at each fixed $\gamma$; it requires no modulus of continuity
uniform in $\gamma$ for the random radial processes.

For a radial density $e^{f_\gamma(t)}\le1$, the comparison applies
directly on each bounded box. Proposition~\ref{mom:whole-strip}
then removes the horizontal cutoff under
\eqref{decomp:radial-envelope}.

\subsection{The unified joint LDP}
\label{ldp:joint-section}

We now prove the joint LDP for each fixed locally absolutely continuous
radial displacement $g:\R\to\R$ with $g(0)=0$ and
$\int_\R|g'(t)|^2\,dt<\infty$. The area is not continuous
on the Gaussian state space, so we first replace it by a mass on a
bounded box with finitely many angular modes. We prove that the random
approximation errors have arbitrarily negative logarithmic rate and
that the deterministic approximations converge uniformly on every
Gaussian-energy sublevel. A single diagonal sequence then satisfies
the extended contraction principle, identifying the area coordinate
without adding any independent cost.

Put $X_\gamma=(\gamma B,\gamma X^\perp)$ and
$\mathcal E=\mathcal B\times E_s$. For fixed $g$ as above, set
\begin{equation}\label{ldp:unshifted-triple}
 A_\gamma^g=Z_\gamma(d_\gamma(\gamma B+g)),\qquad
 Y_\gamma^g=(X_\gamma,A_\gamma^g),\qquad
 A_0^g(u,v)=\int_{\Sstrip}e^{d_0(u+g;t)+v(t,\theta)}\,dt\,d\theta.
\end{equation}
The last expression is defined on finite-energy inputs. The bounds
below show that it is finite and strictly positive there.

\begin{lemma}[Joint LDP with a fixed radial displacement]
\label{ldp:joint}\label{ldp:unshifted-joint}
For each fixed locally absolutely continuous $g$ with $g(0)=0$ and
$g'\in L^2(\R)$, the laws of $Y_\gamma^g$ satisfy a good LDP on
$\mathcal E\times[0,\infty)$ at speed $\gamma^{-2}$, with rate
\begin{equation}\label{ldp:unshifted-joint-cost}
 I_g(u,v,a)=
 \begin{cases}
 I_G(u,v),& I_G(u,v)<\infty,\quad a=A_0^g(u,v),\\
 +\infty,&\text{otherwise}.
 \end{cases}
\end{equation}
These laws are exponentially tight. The deterministic mass $A_0^g$
is continuous on every sublevel of $I_G$.
\end{lemma}

\begin{proof}
\emph{Step 1: continuous approximating maps.}
Let $D_T=[-T,T]\times(0,\pi)$ and let $P_N$ project onto the first
$N$ nonconstant cosine modes. Define
\begin{equation}\label{ldp:finite-map}
 f_{T,N}^g(u,v)=\int_{D_T}e^{d_0(u+g;t)+P_Nv(t,\theta)}\,dt\,d\theta,
 \qquad G_{T,N}^g(u,v)=(u,v,f_{T,N}^g(u,v)).
\end{equation}
For fixed $T,N$, this map is continuous on $\mathcal E$:
the running maximum is locally Lipschitz and the finitely many
coefficients of $P_Nv$ converge uniformly on $[-T,T]$.
The ordinary contraction principle therefore applies to
$G_{T,N}^g(X_\gamma)$.

Write $A_{\gamma,T}^g$ for the restriction of $A_\gamma^g$ to $D_T$,
and $A_{\gamma,T,N}^g$ for its Fourier-cutoff version. To distinguish
this last random variable from $f_{T,N}^g(X_\gamma)$, put
\[
 c_{\gamma,N}(\theta)=(2\sin\theta)^{-\gamma^2/2}
             e^{-\gamma^2G_{\perp,N}((0,\theta),(0,\theta))/2},
 \qquad
 f_{\gamma,T,N}^g(u,v)=
 \int_{D_T}e^{d_\gamma(u+g;t)+P_Nv(t,\theta)}
              c_{\gamma,N}(\theta)\,dt\,d\theta.
\]
Thus $A_{\gamma,T,N}^g=f_{\gamma,T,N}^g(X_\gamma)$, including
the boundary factor and Wick correction. There are three errors:
\begin{align}\label{ldp:three-errors}
 |A_\gamma^g-f_{T,N}^g(X_\gamma)|
 &\le (A_\gamma^g-A_{\gamma,T}^g)
       +|A_{\gamma,T}^g-A_{\gamma,T,N}^g|\notag\\
 &\quad+|f_{\gamma,T,N}^g(X_\gamma)-f_{T,N}^g(X_\gamma)|.
\end{align}

\emph{Step 2: the Fourier error at a fixed box.}
Conditionally on $B$, the density $e^{d_\gamma(\gamma B+g)}$ is
between zero and one and is independent of the lateral field.
Lemma~\ref{mom:cutoff-error} gives numbers
$\varepsilon_{N,T}\to0$ as $N\to\infty$, for each fixed $T$, such that
\[
 \|A_{\gamma,T}^g-A_{\gamma,T,N}^g\|_{L^{q_\gamma}(\mathbb P)}
 \le\varepsilon_{N,T},\qquad
 q_\gamma=1+\frac{W+2}{2\gamma^2}.
\]
For every $\eta>0$, Markov's inequality yields
\[
 \limsup_{\gamma\downarrow0}\gamma^2\log
 \P(|A_{\gamma,T}^g-A_{\gamma,T,N}^g|>\eta)
 \le\frac{W+2}{2}\log\frac{\varepsilon_{N,T}}{\eta}.
\]
The right side tends to $-\infty$ with $N$. The box is fixed
throughout this assertion; uniformity in $T$ is neither claimed nor
needed.

\emph{Step 3: the horizontal tail.}
On $\|\gamma B\|_*\le r$, the envelope
\eqref{ldp:tail-envelope}, the unit-cell bound in
Lemma~\ref{mom:high-moment}, and Minkowski's inequality give a
conditional $L^{q_\gamma}(\mathbb P)$ bound $\varepsilon_T(r,g)$ for
$A_\gamma^g-A_{\gamma,T}^g$. For every fixed $r,g$ this bound tends
to zero as $T\to\infty$, uniformly for sufficiently small $\gamma$:
it is bounded by a constant times the sum of the cell suprema of
the envelope outside $[-T,T]$. Fernique's bound
\eqref{ldp:fernique} and conditional Markov therefore give
\[
 \P(A_\gamma^g-A_{\gamma,T}^g>\eta)
 \le \left(\frac{\varepsilon_T(r,g)}{\eta}\right)^{q_\gamma}
        +C e^{-\epsilon r^2/\gamma^2}.
\]
Here $C,\epsilon>0$ are independent of $T,r,\gamma$. For nonnegative
$p_\gamma,q_\gamma'$, the inequality
$p_\gamma+q_\gamma'\le2\max(p_\gamma,q_\gamma')$ shows that
the logarithmic rate of their sum is at most the maximum of their
logarithmic upper rates. Consequently
\[
 \limsup_{\gamma\downarrow0}\gamma^2\log
 \P(A_\gamma^g-A_{\gamma,T}^g>\eta)
 \le\max\left\{
 \frac{W+2}{2}\log\frac{\varepsilon_T(r,g)}{\eta},
 -\epsilon r^2\right\}.
\]
First let $T\to\infty$ with $r$ fixed, then let $r\to\infty$.
This proves
\begin{equation}\label{ldp:tail-superexponential}
 \lim_{T\to\infty}\limsup_{\gamma\downarrow0}\gamma^2\log
 \P(A_\gamma^g-A_{\gamma,T}^g>\eta)=-\infty.
\end{equation}

\emph{Step 4: the correction at fixed $T,N$.}
The diagonal covariance is bounded for fixed $N$ and is nonnegative.
Thus $c_{\gamma,N}\to1$ in $L^1(0,\pi)$, by dominated convergence
with majorant $\max\{1,(2\sin\theta)^{-1/2}\}$ for $\gamma^2\le1$.
For a compact $K\subset\mathcal E$, let
$M_K=\sup_{(u,v)\in K,z\in D_T}|P_Nv(z)|<\infty$.
Since $d_\gamma,d_0\le0$ and
$|d_\gamma(u+g;t)-d_0(u+g;t)|\le\gamma^2|t|/2$, we have
\[
 \sup_K|f_{\gamma,T,N}^g-f_{T,N}^g|
 \le e^{M_K}\left(2T\|c_{\gamma,N}-1\|_{L^1(0,\pi)}
                       +\pi\gamma^2T^2/2\right)\longrightarrow0.
\]
Given $L>0$, Gaussian exponential tightness supplies such a compact
$K$ with complement of logarithmic upper rate at most $-L$.
For each fixed $\eta>0$, the correction exceeds $\eta$ only outside
$K$ once $\gamma$ is sufficiently small. Since $L$ is arbitrary,
\begin{equation}\label{ldp:finite-correction}
 \limsup_{\gamma\downarrow0}\gamma^2\log
 \P(|f_{\gamma,T,N}^g(X_\gamma)-f_{T,N}^g(X_\gamma)|>\eta)
 =-\infty.
\end{equation}
Both cutoffs remain fixed when taking this limit.

\emph{Step 5: deterministic approximation on energy sublevels.}
Let $K_L=\{I_G\le L\}$, which is compact by
Lemma~\ref{ldp:gaussian}. On $K_L$,
$|u(t)+g(t)|\le(2\sqrt L+\|g'\|_2)\sqrt{|t|}$, so
\[
 d_0(u+g;t)\le -W|t|+(2\sqrt L+\|g'\|_2)\sqrt{|t|}.
\]
The unit-cell exponential bound of Lemma~\ref{var:cell-exponential}
then gives a finite, uniformly summable envelope for the deterministic
area, hence uniformly vanishing horizontal tails. It also gives, on
each fixed $D_T$, uniform bounds on $\int e^{2v}$ and
$\int e^{2P_Nv}$, because projection does not increase the energy.
Moreover,
\[
 \|(1-P_N)v\|_{L^2(\Sstrip)}
 \le(N+1)^{-1}\|v_\theta\|_{L^2(\Sstrip)}.
\]
Using $e^{d_0}\le1$, $|e^x-e^y|\le|x-y|(e^x+e^y)$ and
Cauchy--Schwarz, we obtain
\[
 \sup_{K_L}\left|
 \int_{D_T}e^{d_0(u+g)+v}-f_{T,N}^g(u,v)\right|
 \le\frac{C_{L,T}}{N+1}\longrightarrow0.
\]
These two estimates prove finiteness of $A_0^g$ and permit uniform
approximation on $K_L$, first by choosing $T$, then $N$.
Positivity follows because the integrand is positive almost everywhere.

\emph{Step 6: one sequence for every error threshold.}
We specify the diagonal construction to avoid exchanging limits.
For each integer $j\ge1$ put $\delta_j=1/(3j)$. Choose $T_j\ge j$
so large that the logarithmic upper rate of
$\P(A_\gamma^g-A_{\gamma,T_j}^g>\delta_j)$ is at most $-4j$
and the deterministic horizontal tail is at most $1/(2j)$ on $K_j$.
Steps 3 and 5 ensure that both requirements hold for all sufficiently
large $T_j$. With this $T_j$ fixed, choose $N_j\ge j$ so large that
the logarithmic upper rate of
$\P(|A_{\gamma,T_j}^g-A_{\gamma,T_j,N_j}^g|>\delta_j)$
is at most $-4j$ and the deterministic Fourier error on $D_{T_j}$
is at most $1/(2j)$ on $K_j$. Steps 2 and 5 give these choices.
There is no subsequent enlargement of $T_j$.

Set $G_j=G_{T_j,N_j}^g$. For this fixed pair of cutoffs, Step 4
assigns rate $-\infty$ to the third error at threshold $\delta_j$.
The union bound in \eqref{ldp:three-errors} therefore gives
\begin{equation}\label{ldp:diagonal-error}
 \limsup_{\gamma\downarrow0}\gamma^2\log
 \P\left(|A_\gamma^g-f_{T_j,N_j}^g(X_\gamma)|>1/j\right)
 \le-4j.
\end{equation}
The factor three in the union bound disappears because
$\gamma^2\log3\to0$. Equip the product space with the sum of a
compatible metric on $\mathcal E$ and $1\wedge|a-a'|$ on the mass
coordinate. The first two coordinates of $Y_\gamma^g$ and
$G_j(X_\gamma)$ agree exactly. Given any fixed $\eta>0$, for all
large $j$ the event that their distance exceeds $\eta$ is contained
in the event in \eqref{ldp:diagonal-error}. Hence
\begin{equation}\label{ldp:exponentially-good}
 \lim_{j\to\infty}\limsup_{\gamma\downarrow0}\gamma^2\log
 \P\bigl(d(Y_\gamma^g,G_j(X_\gamma))>\eta\bigr)=-\infty.
\end{equation}
In particular, the sequence of cutoffs does not depend on $\eta$.
Each $j$ is fixed before $\gamma\downarrow0$; no estimate uniform in
all $j$ at a common positive $\gamma$ is required.

Our simultaneous deterministic choices also ensure
\[
 \sup_{K_j}|f_{T_j,N_j}^g-A_0^g|\le1/j.
\]
For every fixed $L$, this proves uniform convergence on $K_L$ as
$j\to\infty$. Since the approximating maps are continuous, $A_0^g$
is continuous on $K_L$.

\emph{Step 7: identification of the joint rate.}
The extended contraction principle for exponentially good continuous
approximations \cite[Theorem 4.2.23]{DZ} applies to the Gaussian LDP,
\eqref{ldp:exponentially-good}, and the uniform convergence on every
energy sublevel. Its rate is
\[
 \inf\{I_G(x): I_G(x)<\infty,\ (x,A_0^g(x))=(u,v,a)\}.
\]
The retained Gaussian coordinates force $x=(u,v)$, so this is
precisely \eqref{ldp:unshifted-joint-cost}. The limiting map need
only be specified on finite-energy inputs; if a map on all of
$\mathcal E$ is required, extend its mass coordinate by zero elsewhere.
This is measurable since the finite-energy set is a countable union
of compact sublevels on each of which $A_0^g$ is continuous.
The rate sublevel of height $L$ is exactly the continuous image
$\{(x,A_0^g(x)):x\in K_L\}$, and is therefore compact.

For completeness, exponential tightness of the mass also follows
directly. On $\|\gamma B\|_*\le r$, Proposition~\ref{mom:whole-strip}
bounds its conditional $L^{q_\gamma}(\mathbb P)$ norm by
$C_W(1+r+\|g\|_*)^4$. Thus
\[
 \P(A_\gamma^g>M)
 \le\left(\frac{C_W(1+r+\|g\|_*)^4}{M}\right)^{q_\gamma}
       +C e^{-\epsilon r^2/\gamma^2}.
\]
For any desired exponential cost choose $r$, then $M$, sufficiently
large. Combining this with Gaussian exponential tightness gives
exponential tightness of $Y_\gamma^g$.
\end{proof}

\subsection{The joint rate in radial-field coordinates}\label{ldp:radial-section}

We now specialize the unified theorem to $g=0$ and contract its
Brownian coordinate to the actual radial field. This gives an explicit
unshifted rate and displays the competition between Gaussian energy
and area reward that Section~\ref{rad:section} must optimize. There is
no new approximation of the GMC in this step.  Let $\mathcal R_W$ consist of
locally absolutely continuous functions $r:\R\to\R$ such that
\[
 r(0)=0,\qquad r\le0,\qquad
 r'+W\operatorname{sgn}\in L^2(\R),
\]
where $\operatorname{sgn}(t)$ is $1$ for $t>0$, $-1$ for $t<0$, and
$0$ at the origin.  Radial profiles carry the local uniform topology.

\begin{corollary}[Joint LDP before Girsanov]
\label{ldp:radial-joint}
Under the original law before area conditioning, the triple
\[
 (\gamma R_\gamma,\gamma X^\perp,Z_\gamma)
\]
satisfies a good LDP of speed $\gamma^{-2}$ on
$C_{\mathrm{loc}}(\R)\times E_s\times[0,\infty)$, with rate
\begin{equation}
 I_{\rm raw}(r,v,a)=
 \begin{cases}
 \displaystyle\frac14\int_\R|r'(t)+W\operatorname{sgn}(t)|^2\dd t
       +\frac1{4\pi}\int_{\Sstrip}|\nabla v|^2\dd^2 z,
 &\begin{gathered}r\in\mathcal R_W,\ v\in\mathcal H_\perp,\\
                    a=Z_0(r+v),\end{gathered}\\[2pt]
 +\infty,&\text{otherwise}.
 \end{cases}
 \label{ldp:raw-rate}
\end{equation}
\end{corollary}

\begin{proof}
Apply Lemma~\ref{ldp:joint} with $g=0$ and contract by the continuous map
$(u,v,a)\mapsto(d_0(u),v,a)$.  The deterministic local bound on
$d_\gamma-d_0$ makes the true reconstruction exponentially equivalent
to this one.  It remains to minimize the Brownian energy at a fixed
$r=d_0(u)$.

For $t\ge0$ put
$m_\sigma(t)=\sup_{0\le s\le t}\{u(\sigma s)-Ws\}$.
Finite energy gives $|u(\sigma t)|\le\sqrt t\,\|u'\|_2$, so each
$m_\sigma$ is bounded and eventually constant.  Running maxima of
locally absolutely continuous functions are locally absolutely
continuous: their increments over an interval are bounded by the
integral of the absolute derivative there.  Their derivative is zero
off the record set, and on that set it agrees almost everywhere with
the derivative of the underlying path.  Thus, where $m_\sigma'>0$,
$m_\sigma'=\frac{d}{dt}u(\sigma t)-W$.  Since
$r(\sigma t)+Wt=u(\sigma t)-2m_\sigma(t)$, this gives
\[
 \left|\frac{d}{dt}u(\sigma t)\right|^2
 -\left|\frac{d}{dt}r(\sigma t)+W\right|^2
 =4Wm_\sigma'(t)
\]
almost everywhere, including off the record set.  Consequently
\[
 \frac14\int_\R|u'|^2\dd t
 =\frac14\int_\R|r'+W\operatorname{sgn}|^2\dd t
      +W\sum_{\sigma=\pm}m_\sigma(\infty).
\]
Every finite-energy preimage therefore has at least the radial cost in
\eqref{ldp:raw-rate}, and its image belongs to $\mathcal R_W$.
Conversely, for $r\in\mathcal R_W$, take $u(t)=r(t)+W|t|$.
This is a finite-energy element of $\mathcal B$.  Since $r\le0$ and
$r(0)=0$, both running maxima vanish, $d_0(u)=r$, and the lower cost
is attained.  The lateral coordinate and mass graph are unchanged.
This proves the rate formula; goodness follows by contraction.
\end{proof}

The unique zero of the unweighted rate is
\[
 (r,v,a)=(-W|t|,0,2\pi/W).
\]
This is the typical triple before conditioning.  It is not the
optimization relevant to the unit-area expectation: that expectation
adds the reward $W\log a$.  In radial-field coordinates the problem
for the next section is to maximize
\[
 W\log Z_0(r+v)
 -\frac14\int_\R|r'+W\operatorname{sgn}|^2\dd t
 -\frac1{4\pi}\int_{\Sstrip}|\nabla v|^2\dd^2 z.
\]
The higher-moment estimates above will justify using this unbounded
reward.  They do not identify its maximizer; that is the variational
step after the Girsanov calculation.

\section{Girsanov and concentration under the unit-area law}
\label{rad:section}\label{proof:section}

The goal of this section is to evaluate the unit-area expectation by
showing that its exact positive weight concentrates on the classical
profile. Subsection~\ref{proof:shift} chooses a radial Cameron--Martin
shift from the stationarity equation and computes its exact Gaussian
density. Subsection~\ref{proof:reduction} inserts this shift into the
normalized expectation and applies the single joint LDP of
Section~\ref{decomp:section}; this identifies the reward minus energy
whose maximum controls the expectation. Subsection~\ref{var:section}
proves that the candidate is its unique maximizer and establishes the
compactness needed for a strict gap away from it.
Subsection~\ref{ldp:weighted-section} uses the high moments to pass from
the probability LDP to weighted Laplace bounds despite the unbounded
area reward. Subsection~\ref{proof:concentration} then controls the
normalizing denominator and proves concentration under the exact
tilted law. Finally, Subsection~\ref{proof:completion} transfers that
concentration to bounded continuous field observables and to the
unit-area measure. Thus the LDP is proved once, and Girsanov serves to
center and solve the ensuing variational problem.

\subsection{Choosing and implementing the radial shift}
\label{proof:shift}

We choose a finite-energy shift that places the candidate optimizer
at the origin of the Gaussian input, and compute the resulting density
exactly. The general Cameron--Martin identity fixes the signs and
variance factors; the radial stationarity equation then singles out
the hyperbolic-cosine profile. This calculation proposes the candidate,
whose global optimality will be proved in Subsection~\ref{var:section}.

\subsubsection{The exact Cameron--Martin identity}

We establish the change-of-variables identity for an arbitrary
finite-energy radial displacement, before choosing that displacement.
The independent lateral field is retained as part of the observable.

Let $g:\R\to\R$ be locally absolutely continuous, with $g(0)=0$
and $\int_\R|g'(t)|^2\dd t<\infty$.  For
$\sigma\in\{-1,1\}$ the stochastic integral
\[
 \int_0^\infty\frac{\dd}{\dd t}g(\sigma t)\,\dd B(\sigma t)
\]
is the $L^2$ limit of its integrals on bounded time intervals,
with respect to the Brownian motion $t\mapsto B(\sigma t)$.
In particular its integrand on the negative half-line is
$-g'(-t)$.

\Needspace{12\baselineskip}
\begin{lemma}[Exact Cameron--Martin identity]
\label{setup:CM}
For every nonnegative measurable functional $G$ of $(B,X^\perp)$,
\begin{align}
 \E G(B,X^\perp)
 ={}&\exp\!\left[-\frac1{4\gamma^2}
                         \int_\R|g'(t)|^2\dd t\right]\notag\\[-2pt]
 &\times\E\!\left[
 G(B+g/\gamma,X^\perp)
 \exp\!\left\{-\frac1{2\gamma}
       \sum_{\sigma\in\{-1,1\}}\int_0^\infty
              \frac{\dd}{\dd t}g(\sigma t)\,\dd B(\sigma t)\right\}
                 \right].
 \label{rad:generic-CM}
\end{align}
\end{lemma}
\begin{proof}
The centered Gaussian functional
$\frac12\sum_\sigma\int_0^\infty
\frac{\dd}{\dd t}g(\sigma t)\,\dd B(\sigma t)$ has covariance
$g(t)$ with $B(t)$ and variance
$\frac12\int_\R|g'(t)|^2\dd t$.
These assertions follow from independence of the two halves and
the It\^o isometry for variance-two Brownian motion.  They identify
the Cameron--Martin representative and its squared norm.
The reverse Gaussian change-of-variables formula for translation
by $g/\gamma$ is therefore \eqref{rad:generic-CM}, including its
negative linear term.  One can first take $g'$ to be a compactly
supported step function and $G$ a bounded cylinder functional;
the identity is then the ordinary finite-dimensional Gaussian
density formula.  Approximation in $L^2$ proves the general
Cameron--Martin identity, and monotone convergence extends it to
nonnegative measurable $G$.  The independent field $X^\perp$
is unchanged.
\end{proof}

\subsubsection{The stationary profile and its Gaussian source}

We seek a candidate for the area-weighted optimizer by restricting
the unshifted variational problem to profiles with no lateral
displacement. We explain this restriction first, then derive its
stationarity equation and use the resulting profile to choose the
Brownian shift. This calculation only selects a candidate: global
optimality among all radial and lateral profiles will be established
in Subsection~\ref{var:section}.

By Corollary~\ref{ldp:radial-joint}, a deterministic radial profile
$q\in\mathcal R_W$ with lateral displacement $v=0$ has area
\[
 Z_0(q)=\int_\R\int_0^\pi e^{q(t)}\,d\theta\,dt
       =\pi\int_\R e^{q(t)}\,dt
\]
and probability cost
$\frac14\int_\R|q'(t)+W\operatorname{sgn}(t)|^2\,dt$.
The area power in \eqref{intro:ratio} is
$p_\gamma=W/\gamma^2-1/2$, so its logarithmic reward at speed
$\gamma^{-2}$ is $W\log Z_0(q)$, since
$\gamma^2p_\gamma\to W$. Thus restricting reward minus the joint
rate to $(q,0,Z_0(q))$ gives
\[
 W\log\left(\pi\int_\R e^{q(t)}\dd t\right)
       -\frac14\int_\R
                    |q'(t)+W\operatorname{sgn}(t)|^2\dd t.
\]
The rate formula already minimizes over Brownian inputs producing
$q$ under reflection: $u(t)=q(t)+W|t|$ attains this cost, because
$q\le0$ and $q(0)=0$ make both running maxima vanish.
The passage from the probability LDP to an expectation with this
unbounded reward will be justified in
Subsection~\ref{ldp:weighted-section}.

To find a candidate, we look within the class of smooth even profiles
$q\le0$, with $q(0)=0$, $q'(0)=0$, limiting slope $-W$ on the
positive half-line, and $q'+W\operatorname{sgn}\in L^2(\R)$.
Here $\operatorname{sgn}(t)$ is $1$ for $t>0$, $-1$ for $t<0$,
and $0$ at zero. Symmetry is an ansatz at this stage, rather than
a conclusion about all optimizers. For a candidate strictly negative
away from zero, compactly supported variations on $(0,\infty)$
preserve the constraint for sufficiently small variation parameter.
Take a smooth variation $\eta$ supported in $(0,\infty)$ and extend
it evenly to the other half-line. One half of the first variation of
this expression is
\[
 \frac{W\int_0^\infty e^{q(t)}\eta(t)dt}{\int_\R e^{q(t)}dt}
 -\frac12\int_0^\infty(q'(t)+W)\eta'(t)dt.
\]
Integration by parts in the second integral, with zero boundary terms
because $\eta$ has compact support, gives the stationarity equation
\[
 q''(t)=-\frac{2W e^{q(t)}}{\int_\R e^{q(s)}\dd s},\qquad t>0.
\]
Multiplying by $2q'$ and using the limiting slope at infinity
gives the first integral
\[
 q'(t)^2=W^2-\frac{4W e^{q(t)}}{\int_\R e^{q(s)}\dd s}.
\]
The initial conditions imply $\int_\R e^q\dd t=4/W$.
The resulting initial-value problem has the unique solution
\begin{equation}
 q(t)=-2\log\cosh\frac{Wt}{2},
 \qquad
 \pi\int_\R\sech^2\frac{Wt}{2}\dd t=\frac{4\pi}{W}.
 \label{rad:profile-choice}
\end{equation}
Thus we choose the Brownian shift
$g_*(t)=W|t|-2\log\cosh(Wt/2)$.  This stationarity calculation neither
linearizes the running maximum nor asserts an exact cancellation
of the lateral first chaos.  Those random terms remain in the
exact expectation.

For $u\in\mathcal B$, define on the positive half-line
\begin{align}
 b_\gamma(u;t)
 ={}&-2\log\cosh\frac{Wt}{2}+u(t)+\frac{\gamma^2t}{2}\notag\\[-2pt]
 &-2\sup_{0\leq s\leq t}
       \left\{-2\log\cosh\frac{Ws}{2}+u(s)+\frac{\gamma^2s}{2}\right\},
 \qquad t\geq0.
 \label{setup:reflected}
\end{align}
For $t<0$, replace $t$ by $|t|$ and $u(s)$ by $u(-s)$ in this
formula.  The notation $b_0$ means the same map with $\gamma=0$.
The supremum contains both zero and the current value, so
$b_\gamma(u;t)\leq0$.  Since the supremum is nonnegative and
$-2\log\cosh(Wt/2)\leq\log4-W|t|$, for $\gamma^2\leq W$ we also have
\begin{equation}
 e^{b_\gamma(u;t)}\leq
 \min\left\{1,4\exp\left[-\frac W2|t|
                    +\|u\|_*(1+|t|)^{3/4}\right]\right\}.
 \label{rad:envelope}
\end{equation}
Thus these reflected profiles satisfy the envelope used in
Section~\ref{decomp:section}, with $r=\|u\|_*$.  The maxima are retained exactly;
no expansion of their location or height is involved.

The covariance and variance identities that determine the Gaussian
source for this shift are
\begin{align}
 \Cov\left(B(t),\frac{W^2}{4}\int_\R
                       B(s)\sech^2\frac{Ws}{2}\dd s\right)
 &=W|t|-2\log\cosh\frac{Wt}{2},
 \label{setup:representer}\\
 \Var\left(\frac{W^2}{4}\int_\R
                       B(t)\sech^2\frac{Wt}{2}\dd t\right)
 &=2W(2\log2-1).
 \label{setup:shift-variance}
\end{align}
For example, for $t\geq0$ the covariance in
\eqref{setup:representer} is
$\frac{W^2}{2}\int_0^\infty(t\wedge s)\sech^2(Ws/2)\dd s$.
Its derivative is $W(1-\tanh(Wt/2))$ and its value at zero is
zero, which proves the identity on this half-line; symmetry
handles the other half.  The squared Cameron--Martin norm is
\begin{align*}
 \frac12\int_\R
  \left|\frac{\dd}{\dd t}
          \left(W|t|-2\log\cosh\frac{Wt}{2}\right)\right|^2\dd t
 &=2W\int_0^\infty(1-\tanh x)^2\dd x\\
 &=2W\int_0^1\frac{1-y}{1+y}\dd y
 =2W(2\log2-1),
\end{align*}
where $y=\tanh x$.  This proves
\eqref{setup:shift-variance} as well.  The source integral exists
in $L^2$: its kernel decays exponentially and
$\|B(t)\|_{L^2}=\sqrt{2|t|}$.

Substituting this shift into Lemma~\ref{setup:CM} gives, for every
nonnegative measurable $G$,
\begin{align}
 \E G(B,X^\perp)
 ={}&e^{-W(2\log2-1)/\gamma^2}\notag\\[-2pt]
 &\times\E\!\left[
 G\left(B+\frac{W|\cdot|-2\log\cosh(W\cdot/2)}\gamma,X^\perp\right)
 e^{-\frac{W^2}{4\gamma}\int_\R
                      B(t)\sech^2(Wt/2)\dd t}\right].
 \label{setup:radial-CM}
\end{align}
Indeed, integration by parts on each half-line gives the exact
identity for its Gaussian score,
\[
 \frac12\sum_{\sigma\in\{-1,1\}}\int_0^\infty
       W\bigl(1-\tanh(Wt/2)\bigr)\,\dd B(\sigma t)
       =\frac{W^2}{4}\int_\R B(t)\sech^2(Wt/2)\dd t.
\]
The term at zero vanishes because $B(0)=0$, and the term at
infinity vanishes by the exponential decay of the derivative
of the shift.  Equivalently,
\eqref{setup:representer} identifies this Gaussian functional
directly.

\subsection{The exact weighted expectation and its variational problem}
\label{proof:reduction}

We now express the unit-area expectation under one explicitly defined
tilted law and identify its limiting reward minus Gaussian energy.
The shift $g_*$ is already covered by Lemma~\ref{ldp:joint}; applying
that theorem gives the joint field--area rate directly, with no second
LDP proof and no separate lateral change of measure. This reduces the
desired observable limit to concentration of a specified triple.
Recall $\Phi_\gamma=\gamma h_\gamma^{(1)}$ and $b_\gamma$ from
\eqref{setup:reflected}.

\begin{lemma}[Exact shifted expectation]
\label{setup:shifted-identity}
Fix $s>0$ and let
$\widehat E_s=C_{\mathrm{loc}}(\R;H_N^{-s}(0,\pi))$
with its local uniform topology, as in \eqref{ldp:spaces}.
For every bounded Borel measurable $F:\widehat E_s\to\R$
and all sufficiently small $\gamma$,
\begin{equation}
\label{eq:shifted-expectation}
 \E^{(1)}[F(\Phi_\gamma)]=
 \frac{\begin{aligned}
  \E\!\Bigg[&F\!\left(b_\gamma(\gamma B)+\gamma X^\perp
                    -\log Z_\gamma(b_\gamma(\gamma B))\right)\\[-2pt]
  &\times\left(\frac W{4\pi}Z_\gamma(b_\gamma(\gamma B))\right)^{W/\gamma^2-1/2}
  e^{-\frac{W^2}{4\gamma^2}\int_\R
                         \gamma B(t)\sech^2(Wt/2)\dd t}\Bigg]
 \end{aligned}}
 {\displaystyle\E\!\left[
  \left(\frac W{4\pi}Z_\gamma(b_\gamma(\gamma B))\right)^{W/\gamma^2-1/2}
  e^{-\frac{W^2}{4\gamma^2}\int_\R
                         \gamma B(t)\sech^2(Wt/2)\dd t}\right]}.
\end{equation}
The same identity holds for bounded Borel observables on
$H^{-s}_{\mathrm{loc}}(\Sstrip)$, by composing with the continuous
restriction map from $\widehat E_s$. No continuity or Lipschitz
assumption is needed for this exact identity. For the limit of
expectations in Subsection~\ref{proof:completion}, we take $F$
bounded and continuous; Lipschitz continuity is not required.
\end{lemma}
\begin{proof}
Insert Lemma~\ref{setup:Pitman} into the original unit-area
expectation \eqref{intro:ratio}.  Under the translation in
\eqref{setup:radial-CM}, the expression inside its running
supremum, multiplied by $\gamma$, becomes
\[
 \gamma B(s)+W s-2\log\cosh(Ws/2)-(W-\gamma^2/2)s
       =-2\log\cosh(Ws/2)+\gamma B(s)+\gamma^2s/2.
\]
Consequently the entire scaled reflected radial process becomes
$b_\gamma(\gamma B)$, exactly as defined in
\eqref{setup:reflected}.  The mass becomes
$Z_\gamma(b_\gamma(\gamma B))$, and the same substitution in
$F$ gives the displayed field argument, including the logarithm
of the shifted mass.  The Gaussian density contributes the
displayed source factor.  The deterministic factor
$e^{-W(2\log2-1)/\gamma^2}$ cancels between numerator and
denominator; multiplying both integrands by
$(W/(4\pi))^{W/\gamma^2-1/2}$ then gives
\eqref{eq:shifted-expectation}.

The changes of measure hold first for nonnegative observables.
To justify normalization and subtraction for bounded signed
$F$, Proposition~\ref{mom:whole-strip}, conditional on $B$, bounds
the required mass moment by a constant times
$(1+\gamma\|B\|_*)^{4W/\gamma^2-2}$ for each fixed sufficiently
small $\gamma$.  The absolute source is bounded by a constant
depending on $W$ times $\|B\|_*/\gamma$.
The Gaussian exponential integrability of the Brownian norm
therefore makes both expectations finite.  The denominator is
strictly positive by positivity of the mass.  The Gaussian path
space and its exponential integrability were verified in
Subsection~\ref{ldp:gaussian-section}.
\end{proof}

Define the probability law $\mathbb Q_\gamma$ by
the following fully specified Radon--Nikodym derivative:
\begin{equation}\label{ldp:tilted-law}
 \frac{d\mathbb Q_\gamma}{d\mathbb P}
 =\frac{\displaystyle
 \left(\frac W{4\pi}Z_\gamma(b_\gamma(\gamma B))\right)^{W/\gamma^2-1/2}
 e^{-\frac{W^2}{4\gamma^2}\int_\R\gamma B(t)\sech^2(Wt/2)\,dt}}
 {\displaystyle\E\left[
 \left(\frac W{4\pi}Z_\gamma(b_\gamma(\gamma B))\right)^{W/\gamma^2-1/2}
 e^{-\frac{W^2}{4\gamma^2}\int_\R\gamma B(t)\sech^2(Wt/2)\,dt}
 \right]}.
\end{equation}
The preceding proof gives a finite positive denominator. This law is
exactly the normalized weight in \eqref{eq:shifted-expectation}; it is
defined now so that every subsequent concentration statement refers
to a specified probability measure. Lemma~\ref{ldp:weight-moment}
will give the stronger bound needed at exponential scale.

The shift changes the deterministic radial map, but not the Gaussian
pair whose law is used on the right side of
\eqref{eq:shifted-expectation}.  With
$g_*(t)=W|t|-2\log\cosh(Wt/2)$, one has exactly
\[
 b_\gamma(u)=d_\gamma(u+g_*).
\]
To apply that theorem to the exact expectation, define the shifted
mass and the joint random variable on the original Gaussian
probability space $(\Omega,\mathbb P)$ by
\begin{equation}\label{ldp:triple}
 A_\gamma^*=Z_\gamma(b_\gamma(\gamma B)),\qquad
 Y_\gamma^*=(\gamma B,\gamma X^\perp,A_\gamma^*).
\end{equation}
Since $g_*\in\mathcal B$ and $g_*'\in L^2(\R)$,
Lemma~\ref{ldp:joint} with $g=g_*$ says that the laws of
$Y_\gamma^*$ under $\mathbb P$ satisfy a good LDP, with exponential
tightness, on $\mathcal B\times E_s\times[0,\infty)$.
Its rate under $\mathbb P$ is
\begin{equation}\label{ldp:joint-cost}
 I_*(u,v,a)= \begin{cases}
 \displaystyle\frac14\int_\R|u'|^2
 +\frac1{4\pi}\int_{\Sstrip}|\nabla v|^2,
 &\begin{gathered}u'\in L^2(\R),\ v\in\mathcal H_\perp,\\
                  a=Z_0(b_0(u)+v),\end{gathered}\\[2pt]
 +\infty,&\text{otherwise}.
 \end{cases}
\end{equation}
As in Section~\ref{decomp:section}, finite radial energy includes local
absolute continuity and $u(0)=0$.
This rate describes the centered Gaussian law $\mathbb P$, not the
tilted law $\mathbb Q_\gamma$. The latter weights each outcome by its
area and Gaussian source according to \eqref{ldp:tilted-law}; we now
identify the resulting variational functional.

From here on, every use of Lemma~\ref{ldp:joint} in the shifted
variables means its specialization to $g=g_*$. Define
\[
 K(u)=\frac{W^2}{4}\int_\R u(t)\sech^2(Wt/2)\,dt,\qquad
 \Psi_\gamma(u,a)=(W-\gamma^2/2)\log\frac{Wa}{4\pi}-K(u).
\]
The unnormalized density in \eqref{ldp:tilted-law} is
$\exp\{\Psi_\gamma(u,a)/\gamma^2\}$ evaluated at the triple
\eqref{ldp:triple}. Its logarithmic limit is
$\Psi(u,a)=W\log(Wa/(4\pi))-K(u)$ for $a>0$.
In particular, the random mass has not been discarded: it is the
third coordinate of $Y_\gamma^*$, and the logarithmic area reward is
evaluated at that coordinate. The joint LDP assigns finite cost only
on the graph $a=Z_0(b_0(u)+v)$. On this graph, subtracting the reward
from the probability cost gives the candidate functional
\begin{equation}\label{proof:objective}
 J(u,v)=I_G(u,v)+K(u)
          -W\log\left[\frac W{4\pi}Z_0(b_0(u)+v)\right].
\end{equation}
Thus $-J(u,v)=\Psi(u,a)-I_*(u,v,a)$ on the finite-rate graph,
which is exactly the expression in \eqref{intro:variational-expression}.
The appearance of $Z_0$ in \eqref{proof:objective} comes from this
graph constraint, not from substituting a limit for $Z_\gamma$ inside
the expectation. Section~\ref{decomp:section} has already justified
the joint LDP for the random mass, including its cutoff limits.
What remains is to justify applying that LDP with the unbounded
weight $\exp(\Psi_\gamma/\gamma^2)$. The high-moment estimate and
weighted upper and lower bounds in
Subsection~\ref{ldp:weighted-section} will do precisely this.
The next subsection first solves the deterministic problem by proving
$J\ge0$ with equality only at $(u,v)=(0,0)$.

The resulting concentration statement we need is
\begin{equation}\label{intro:triple-target}
 (\gamma B,\gamma X^\perp,Z_\gamma(b_\gamma(\gamma B)))
 \longrightarrow(0,0,4\pi/W)
 \quad\text{under }\mathbb Q_\gamma.
\end{equation}
Indeed, the exact field argument of the observable is
\begin{equation}\label{intro:F-argument}
 b_\gamma(\gamma B)+\gamma X^\perp
                   -\log Z_\gamma(b_\gamma(\gamma B)).
\end{equation}
At the limiting triple it equals
$-2\log\cosh(Wt/2)-\log(4\pi/W)=\phi_W$.
To justify this conclusion for expectations, we still need a strict
variational gap, control of the unbounded weight, and a nonvanishing
denominator at exponential scale. These are the tasks of the next
three subsections.

\subsection{The joint radial--lateral variational problem}
\label{var:section}

We prove that $J$ in \eqref{proof:objective} has unique minimum zero
at $(u,v)=(0,0)$, and that near-minimizers cannot escape to infinity
in Gaussian energy. Reflection first converts the running maximum
into an explicit nonnegative cost; the sharp Onofri inequality then
identifies equality. A separate mass bound gives coercivity and
compactness, which turn uniqueness into a strict gap away from the
candidate. That gap will force concentration under the area weight.

Throughout this subsection $u$ is locally absolutely continuous on $\R$,
$u(0)=0$, and $\int_\R|u'|^2<\infty$; the lateral function
$v\in\mathcal H_\perp$ has zero angular mean.  In particular
\begin{equation}
\label{var:poincare}
 \int_{\Sstrip}|v|^2\dd^2 z\le\int_{\Sstrip}|v_\theta|^2\dd^2 z.
\end{equation}
To express the reflection in $b_0$, define, for $t\ge0$,
\begin{equation}
\label{var:reflection-definitions}
 m_\pm(t)=\sup_{0\le s\le t}
       \{-2\log\cosh(Ws/2)+u(\pm s)\},\qquad
 f(\pm t)=u(\pm t)-2m_\pm(t).
\end{equation}
Since $|u(t)|\le\sqrt{|t|}\,\|u'\|_2$, the expression inside each
supremum tends to $-\infty$.  Thus $m_\pm(\infty)$ is finite and
$m_\pm$ is eventually constant.  Also
\begin{equation}
\label{var:constraints}
 f(0)=0,\qquad
 b_0(u;t)=-2\log\cosh(Wt/2)+f(t)\le0.
\end{equation}
Their differentiation properties will be verified in the reflection proof.

\subsubsection{Reflection and uniqueness}

We account exactly for the energy and source terms introduced by the
running maxima. Applying Onofri to the reflected perturbation then
proves $J\ge0$ and identifies its unique zero, including the pinning
of the radial maximum.

\Needspace{16\baselineskip}
\begin{lemma}
\label{var:reflection}
The function $f$ in \eqref{var:reflection-definitions} has finite
Dirichlet energy, and
\begin{align}
 \frac14\int_\R|f'|^2\dd t
  +\frac{W^2}{4}\int_\R f(t)\sech^2(Wt/2)\dd t
 &=\frac14\int_\R|u'|^2\dd t
  +\frac{W^2}{4}\int_\R u(t)\sech^2(Wt/2)\dd t\notag\\
 &\quad-W\sum_{\sigma=\pm}m_\sigma(\infty).
 \label{var:reflection-formula}
\end{align}
Furthermore,
\begin{equation}
\label{var:energy-comparison}
 \int_\R|u'|^2\dd t
 \le\int_\R|f'|^2\dd t
       +4W\sum_{\sigma=\pm}m_\sigma(\infty).
\end{equation}
\end{lemma}

\begin{proof}
Work on either half-line, differentiating with respect to its
nonnegative coordinate.  For a locally absolutely continuous function $q$, its running
maximum $m(t)=\sup_{0\le s\le t}q(s)$ is locally absolutely
continuous: for $a<b$, $0\le m(b)-m(a)\le\int_a^b|q'(s)|ds$.
The latter inequality also applies to disjoint intervals and proves
absolute continuity. On the open set $q<m$, $m$ is locally constant.
On the set $q=m$, their derivatives agree almost everywhere, since
the derivative of an absolutely continuous function vanishes almost
everywhere on each of its level sets, applied to $q-m$. Thus whenever
$m_\sigma'>0$, outside a null set,
\[
 m_\sigma'(t)=-W\tanh(Wt/2)
                  +\frac{\dd}{\dd t}u(\sigma t)>0.
\]
Thus $0\le m_\sigma'\le|\frac{\dd}{\dd t}u(\sigma t)|$, which proves
square integrability of the reflected derivative.  On and off the
record set the same algebra gives
\[
 \left|\frac{\dd}{\dd t}f(\sigma t)\right|^2
 -\left|\frac{\dd}{\dd t}u(\sigma t)\right|^2
 =-4W\tanh(Wt/2)m_\sigma'(t).
\]
The difference of the source terms on this half-line is
$-\frac{W^2}{2}\int_0^\infty m_\sigma(t)\sech^2(Wt/2)\dd t$.
Adding it to one quarter of the preceding identity and integrating
gives
\[
 \int_0^\infty\frac{\dd}{\dd t}
       \{-W\tanh(Wt/2)m_\sigma(t)\}\dd t
 =-Wm_\sigma(\infty).
\]
All terms are integrable: $m_\sigma$ is bounded and eventually constant,
and $\sech^2(Wt/2)$ decays exponentially.  Summing proves
\eqref{var:reflection-formula}.  The derivative identity and
$0\le\tanh(Wt/2)\le1$ prove \eqref{var:energy-comparison}.
\end{proof}

\begin{lemma}
\label{var:selection}
For $0<W<2$, the deterministic mass $Z_0(b_0(u)+v)$ is positive and
finite, and
\begin{align}
 &\frac14\int_\R|u'|^2\dd t
   +\frac1{4\pi}\int_{\Sstrip}|\nabla v|^2\dd^2 z
   +\frac{W^2}{4}\int_\R u(t)\sech^2(Wt/2)\dd t
   -W\log\left(\frac{W}{4\pi}Z_0(b_0(u)+v)\right)
 \notag\\
 &\hspace{8mm}\ge
 W\sum_{\sigma=\pm}m_\sigma(\infty)
   +\frac{4-W^2}{16\pi}\int_{\Sstrip}|v_\theta|^2\dd^2 z
 \ge0.
 \label{var:selection-formula}
\end{align}
The left-hand side vanishes if and only if
$u=0$ and $v=0$.
\end{lemma}

\begin{proof}
Apply Lemma~\ref{var:onofri} to $g=f+v$.  Its hypotheses hold since
$|f(t)|\le\sqrt{|t|}\,\|f'\|_2$, the reference density decays
exponentially, and \eqref{var:poincare} controls $v$ in $L^2$.
This also proves finiteness of the mass; positivity is immediate.
Zero angular mean gives
\[
 W\int_{\Sstrip}(f+v)e^{\phi_W}\dd^2 z
  =\frac{W^2}{4}\int_\R f(t)\sech^2(Wt/2)\dd t,
 \qquad
 \int_{\Sstrip}|f'+v_t|^2\dd^2 z
   =\pi\int_\R|f'|^2\dd t+\int_{\Sstrip}|v_t|^2\dd^2 z.
\]
The second identity follows by differentiating the angular mean in the
weak sense; its mixed term is integrable by Cauchy--Schwarz.  Substituting
these identities into \eqref{var:onofri-formula} and then using
\eqref{var:reflection-formula} proves \eqref{var:selection-formula}.

If the first expression is zero, $W<2$ forces $v_\theta=0$, hence
$v=0$ by its angular mean.  Both running maxima also vanish.  Equality
then holds in radial Onofri, so $f$ has form
\eqref{var:onofri-equality}.  The condition $f(0)=0$ gives $c=0$.
The smooth function $-2\log\cosh(Wt/2)+f(t)$ has maximum zero at the
origin by \eqref{var:constraints}.  Its derivative there is $-W\kappa$,
so $\kappa=0$.  Thus $f=u=0$.  Conversely the zero pair gives
$Z_0(b_0(0))=4\pi/W$ and equality directly.
\end{proof}

\subsubsection{Coercivity and compactness}

The last inequality controls the angular lateral energy.  We now combine
the cell estimate from Section~\ref{decomp:section} with the reflected
radial envelope to obtain coercivity of the full tilted rate.

\begin{lemma}\label{var:mass-bound-lemma}
For the reflected perturbation $f$ defined in \eqref{var:reflection-definitions},
\begin{equation}
\label{var:mass-bound}
 Z_0(b_0(u)+v)
 \lesssim_W\left(1+\int_\R|f'|^2\dd t\right)
       \exp\left\{\frac1{8\pi}
                         \int_{\Sstrip}|\nabla v|^2\dd^2 z\right\}.
\end{equation}
\end{lemma}

\begin{proof}
First, \eqref{var:constraints} and Cauchy--Schwarz give
\begin{equation}
\label{var:radial-envelope}
 b_0(u;t)\le\min\left\{0,-W|t|+2\log2+
              \sqrt{|t|}\left(\int_\R|f'|^2\dd t\right)^{1/2}\right\}.
\end{equation}
Outside $|t|\le4W^{-2}\int|f'|^2$, the square-root term is at most
$W|t|/2$.  Summing the unit-cell suprema therefore gives
\[
 \sum_{j\in\mathbb Z}\sup_{t\in[j,j+1]}e^{b_0(u;t)}
       \lesssim_W1+\int_\R|f'|^2\dd t.
\]
Multiply \eqref{var:cell-exponential-formula} with $q=1$ by these suprema and sum.
Tonelli's theorem gives \eqref{var:mass-bound}.
\end{proof}

\begin{lemma}
\label{var:coercivity}
For fixed $W\in(0,2)$,
\begin{align}
 &\frac14\int_\R|u'|^2\dd t
   +\frac1{4\pi}\int_{\Sstrip}|\nabla v|^2\dd^2 z
   +\frac{W^2}{4}\int_\R u(t)\sech^2(Wt/2)\dd t
   -W\log\left(\frac W{4\pi}Z_0(b_0(u)+v)\right)
 \notag\\
 &\hspace{8mm}\ge
 \frac18\int_\R|f'|^2\dd t
  +\frac{2-W}{8\pi}\int_{\Sstrip}|\nabla v|^2\dd^2 z
  +W\sum_{\sigma=\pm}m_\sigma(\infty)-O_W(1).
 \label{var:coercivity-formula}
\end{align}
Thus its bounded sublevels have bounded total Gaussian energy.
\end{lemma}

\begin{proof}
Use \eqref{var:reflection-formula} to replace $u$ in the energy and
source by $f$, adding $W\sum_\sigma m_\sigma(\infty)$.
Equation \eqref{var:mass-bound} bounds the mass term from below by
\[
 -\frac W{8\pi}\int_{\Sstrip}|\nabla v|^2\dd^2 z
 -W\log\left(1+\int_\R|f'|^2\dd t\right)-O_W(1).
\]
Since $f(0)=0$, Cauchy--Schwarz bounds the absolute source by
\[
 \frac{W^2}{4}\|f'\|_2
          \int_\R\sqrt{|t|}\sech^2(Wt/2)dt.
\]
The displayed integral is finite. The elementary inequality
$ab\le a^2/16+4b^2$ bounds this term by
$\|f'\|_2^2/16+O_W(1)$.
Also $W\log(1+x)\le x/16+O_W(1)$ for $x\ge0$: the continuous
function $W\log(1+x)-x/16$ tends to $-\infty$ and therefore has
a finite maximum. Apply this with $x=\|f'\|_2^2$.
Subtracting these two bounds from the coefficient $1/4$ leaves
$1/8$, proving the displayed estimate.  Finally \eqref{var:energy-comparison} controls the original
radial energy as well.
\end{proof}

We use the radial weighted space from \eqref{setup:path-space} and the
lateral space $E_s$ from \eqref{setup:spaces}.

\begin{lemma}
\label{var:compactness}
On bounded Gaussian-energy sets, the map
$(u,v)\mapsto Z_0(b_0(u)+v)$ is continuous for convergent subsequences
in the radial weighted topology and $E_s$, $s>0$.
The expression on the left of \eqref{var:coercivity-formula} has compact
sublevels in these topologies.
\end{lemma}

\begin{proof}
Lemma~\ref{ldp:joint} with $g=g_*$ proves continuity of
$Z_0(b_0(u)+v)$ on every Gaussian-energy sublevel. This mass is
strictly positive there, so its logarithm is continuous as well.
The source $K$ is continuous on $\mathcal B$, since
\[
 |K(u)-K(\widetilde u)|
 \le\frac{W^2}{4}\|u-\widetilde u\|_*
       \int_\R(1+|t|)^{3/4}\sech^2(Wt/2)\,dt.
\]
By Lemma~\ref{var:coercivity}, each sublevel of $J$ lies in a
fixed compact sublevel of $I_G$. On this latter set, lower
semicontinuity of $I_G$ and continuity of the mass and source make
$J$ lower semicontinuous. Its sublevel is therefore a closed subset
of a compact set, proving the claim.
\end{proof}

\begin{lemma}\label{var:strict-gap}
For every open neighborhood $U$ of $(0,0)$,
\begin{align}
 \inf_{(u,v)\notin U}\biggl\{
  &\frac14\int_\R|u'|^2\dd t
   +\frac1{4\pi}\int_{\Sstrip}|\nabla v|^2\dd^2 z
   +\frac{W^2}{4}\int_\R u(t)\sech^2(Wt/2)\dd t
 \notag\\[-2mm]
  &\hspace{20mm}-W\log\left(\frac W{4\pi}Z_0(b_0(u)+v)\right)
                      \biggr\}>0.
 \label{var:gap}
\end{align}
\end{lemma}
\begin{proof}
If \eqref{var:gap} failed, Lemma~\ref{var:compactness} would give
a convergent subsequence of pairs outside $U$ whose costs tend to zero.
Lower semicontinuity would give a limit of cost zero.  Lemma~\ref{var:selection}
identifies that limit as $(0,0)$, contradicting closedness of the
complement of $U$.
\end{proof}

The strict gap is the deterministic input that will suppress every
closed set of fluctuations separated from zero in the normalized
expectation.  The moment and approximation estimates of
Section~\ref{decomp:section} allow us to apply it to the Gaussian pair
together with its total chaos mass.

\subsection{The weighted Laplace principle}
\label{ldp:weighted-section}

Our goal is to convert the joint probability LDP into upper and lower
bounds for the weighted expectations in \eqref{eq:shifted-expectation}.
First a moment of the full weight of some fixed order $q>1$ controls
its contribution outside compact sets. On compact sets, we remove
small masses and use the ordinary LDP with the continuous logarithmic
reward. This justifies optimizing $-J$ for both numerator and
denominator, despite the unbounded area power and Gaussian source.

\subsubsection{The unbounded weight in the exact expectation}

The weight in \eqref{eq:shifted-expectation} is not bounded. The
following estimate supplies precisely the integrability needed to
retain it in the LDP. Fix any real $q$ with
$1<q<(W+2)/(2W)$.

\begin{lemma}\label{ldp:weight-moment}
One has
\begin{align}\label{ldp:weight-moment-eq}
 \limsup_{\gamma\to0}\gamma^2\log\E\Bigg[
 &\left(\frac W{4\pi}Z_\gamma(b_\gamma(\gamma B))\right)^%
 {q(W/\gamma^2-1/2)}\notag\\[-2pt]
 &\quad\times\exp\left\{-\frac{qW^2}{4\gamma^2}
 \int_\R\gamma B(t)\sech^2(Wt/2)\,dt\right\}\Bigg]<\infty.
\end{align}
\end{lemma}

\begin{proof}
For sufficiently small $\gamma$,
$0<q(W/\gamma^2-1/2)\le1+(W+2)/(2\gamma^2)$.
Condition on $\gamma B=u$ and apply Proposition~\ref{mom:whole-strip} and
Lyapunov's inequality. Since
\[
 \left|\int_\R u(t)\sech^2(Wt/2)\,dt\right|
 \le\|u\|_*\int_\R(1+|t|)^{3/4}\sech^2(Wt/2)\,dt,
\]
the logarithm of the conditional expectation in
\eqref{ldp:weight-moment-eq}, multiplied by $\gamma^2$, is bounded above
by
\[
 q(W-\gamma^2/2)\bigl(O_W(1)+4\log(1+\|u\|_*)\bigr)
 +\frac{qW^2}{4}\|u\|_*
          \int_\R(1+|t|)^{3/4}\sech^2(Wt/2)dt.
\]
It is therefore $O_{W,q}(1+\|u\|_*)$.
Choose $\epsilon>0$ satisfying \eqref{ldp:fernique}.
Young's inequality bounds this last expression by
$\epsilon\|u\|_*^2+O_{W,q,\epsilon}(1)$ for every $u$.
After substituting $u=\gamma B$, the required expectation is at most
\[
 \exp\{O_{W,q,\epsilon}(1)/\gamma^2\}
                  \E e^{\epsilon\|B\|_*^2}.
\]
The second factor is finite and independent of $\gamma$. Taking
logarithms and multiplying by $\gamma^2$ proves the claim.
\end{proof}

\subsubsection{Weighted upper and lower bounds}

We prove the closed-set upper bound and open-set lower bound for the
full weight. Exponential tightness and the preceding moment estimate
remove large excursions, while a separate estimate near zero area
handles the singular logarithm. The remaining compact-set argument
uses Lemma~\ref{ldp:joint} with $g=g_*$ directly.

\begin{lemma}\label{ldp:weighted-upper}
For every closed $\mathcal C\subset\mathcal B\times E_s\times[0,\infty)$,
\begin{align}
 \limsup_{\gamma\to0}\gamma^2\log\E\Bigg[
 &\ind_{\mathcal C}\bigl(\gamma B,\gamma X^\perp,
                  Z_\gamma(b_\gamma(\gamma B))\bigr)
 \left(\frac W{4\pi}Z_\gamma(b_\gamma(\gamma B))\right)^{W/\gamma^2-1/2}
 \notag\\[-2pt]
 &\qquad\times
 e^{-\frac{W^2}{4\gamma^2}\int_\R\gamma B(t)\sech^2(Wt/2)\,dt}
 \Bigg]\notag\\
 \le\sup_{\substack{(u,v,a)\in\mathcal C\\a=Z_0(b_0(u)+v)}}
 \Bigg\{&W\log\frac{Wa}{4\pi}
 -\frac{W^2}{4}\int_\R u(t)\sech^2(Wt/2)\,dt
 -\frac14\int_\R|u'|^2
 -\frac1{4\pi}\int_{\Sstrip}|\nabla v|^2\Bigg\}.
 \label{ldp:weighted-upper-eq}
\end{align}
The supremum is over finite-energy pairs. For open $\mathcal C$,
the corresponding statement holds with $\liminf$ and the opposite
inequality.
\end{lemma}

\begin{proof}
By exponential tightness in Lemma~\ref{ldp:joint}, for each $T>0$
there is a compact set outside which the triple has logarithmic
probability at most $-T$. Holder's inequality with exponents
$q,q/(q-1)$ and Lemma~\ref{ldp:weight-moment} bound the weighted
contribution of that complement by logarithmic rate
\[
 \frac1q\limsup_{\gamma\to0}\gamma^2\log
 \E\left[\left(\frac W{4\pi}Z_\gamma(b_\gamma(\gamma B))\right)^%
 {q(W/\gamma^2-1/2)}
 e^{-\frac{qW^2}{4\gamma^2}\int_\R\gamma B(t)\sech^2(Wt/2)\,dt}
 \right]-\left(1-\frac1q\right)T.
\]
This tends to $-\infty$ as $T\to\infty$.

On a fixed such compact set, the radial norms are bounded.
For $0<a\le\epsilon<1$ and sufficiently small $\gamma$,
\begin{align*}
 &(W-\gamma^2/2)\log\frac{Wa}{4\pi}
 -\frac{W^2}{4}\int_\R u(t)\sech^2(Wt/2)\,dt\\
 &\hspace{30mm}\le\frac W2\log\epsilon+O(1),
\end{align*}
where the bound is uniform on that compact set. This removes
$a$ close to zero by letting $\epsilon\downarrow0$.
On the remaining compact set $a\ge\epsilon$, the logarithm of the
weight multiplied by $\gamma^2$ converges uniformly to
\[
 W\log\frac{Wa}{4\pi}
 -\frac{W^2}{4}\int_\R u(t)\sech^2(Wt/2)\,dt.
\]
Cover the compact set by finitely many neighborhoods on which this
continuous expression oscillates by at most any prescribed positive
amount. Bound the weight on each neighborhood by its supremum and
apply the LDP upper bound to the closed intersections with the
compact set and $\mathcal C$. The logarithmic upper bound for this
finite sum is the largest of the corresponding bounds. Letting the
oscillation tend to zero, then $\epsilon\downarrow0$, and finally
$T\to\infty$ proves \eqref{ldp:weighted-upper-eq}.
For the lower bound, fix a finite-rate point $(u,v,a)$ in the open
set $\mathcal C$. Here $a>0$. Given $\eta>0$, continuity gives an
open neighborhood $V\subset\mathcal C$ on which the limiting
logarithmic weight differs from its value at $(u,v,a)$ by less than
$\eta$, and the mass coordinate lies in $(a/2,2a)$.
The difference between the exact logarithmic weight and its limit is
$-(\gamma^2/2)\log(Wa'/(4\pi))$ at mass $a'$, which tends to zero
uniformly on this interval. Restrict the expectation to $V$ and
bound the weight below by the exponential of its value at the
chosen point minus $2\eta$, divided by $\gamma^2$.
The probability LDP lower bound for $V$ is at least the negative
Gaussian energy of $(u,v)$, because $V$ contains that point.
Let $\eta\downarrow0$ and take the supremum over finite-rate points
in $\mathcal C$. This proves the stated open-set bound.
\end{proof}

\subsection{Normalization and concentration under the tilted law}
\label{proof:concentration}

We now show that the exact tilted law concentrates exponentially at
$(0,0,4\pi/W)$. The weighted Laplace bounds and the unique optimizer
first give logarithmic rate zero for the denominator. Applying the
same bounds off any neighborhood of the optimizer, using the strict
variational gap, then gives an exponentially small normalized
probability. Keeping the deterministic Girsanov factors also recovers
the original partition function's leading free energy.

\begin{lemma}\label{ldp:normalizer}
The denominator in \eqref{eq:shifted-expectation} satisfies
\begin{equation}\label{ldp:normalizer-eq}
 \lim_{\gamma\to0}\gamma^2\log\E\left[
 \left(\frac W{4\pi}Z_\gamma(b_\gamma(\gamma B))\right)^{W/\gamma^2-1/2}
 e^{-\frac{W^2}{4\gamma^2}\int_\R\gamma B(t)\sech^2(Wt/2)\,dt}
 \right]=0.
\end{equation}
\end{lemma}

\begin{proof}
The deterministic inequality and uniqueness statement in
Lemma~\ref{var:selection} make the supremum in
\eqref{ldp:weighted-upper-eq}, over the whole space, equal to zero.
It is attained at $(u,v,a)=(0,0,4\pi/W)$. This gives the upper bound.
In an open neighborhood of that point the logarithmic weight
multiplied by $\gamma^2$ is bounded below by an arbitrarily small
negative number, for all sufficiently small $\gamma$. The
probability of the neighborhood has lower logarithmic rate zero
by Lemma~\ref{ldp:joint}. Restricting the expectation to it gives
the lower bound.
\end{proof}

\begin{lemma}\label{ldp:concentration}
For every open neighborhood $U$ of $(0,0,4\pi/W)$ in
$\mathcal B\times E_s\times[0,\infty)$,
\begin{equation}\label{ldp:concentration-eq}
 \limsup_{\gamma\to0}\gamma^2\log\mathbb Q_\gamma
 \left(\bigl(\gamma B,\gamma X^\perp,
 Z_\gamma(b_\gamma(\gamma B))\bigr)\notin U\right)<0.
\end{equation}
\end{lemma}

\begin{proof}
Apply Lemma~\ref{ldp:weighted-upper} to the closed complement of
$U$. Its supremum is strictly negative. Indeed, if a sequence of
finite-energy pairs in this complement made the expression tend to
zero, Lemma~\ref{var:coercivity} would bound their Gaussian
energies. Lemma~\ref{ldp:gaussian} and the continuity on energy
sublevels proved in Lemma~\ref{ldp:joint} would give a convergent
subsequence of the triples. Lower semicontinuity of the energies
and continuity of the logarithmic mass and source would make its
limit another optimizer. Uniqueness in Lemma~\ref{var:selection} forces the
limit to be $(0,0,4\pi/W)$, contradicting the closed constraint.
Subtracting the denominator rate \eqref{ldp:normalizer-eq} proves
\eqref{ldp:concentration-eq}.
\end{proof}

Keeping the deterministic factors instead of cancelling them
gives the exact partition-function identity
\begin{align}
 &\E\!\left[Z_\gamma(\gamma R_\gamma)^{W/\gamma^2-1/2}\right]
 =\left(\frac W{4\pi}\right)^{1/2}
     e^{\frac W{\gamma^2}(1+\log(\pi/W))}\notag\\[-2pt]
 &\qquad\times\E\!\left[
 \left(\frac W{4\pi}Z_\gamma(b_\gamma(\gamma B))\right)^{W/\gamma^2-1/2}
 e^{-\frac{W^2}{4\gamma^2}\int_\R
                    \gamma B(t)\sech^2(Wt/2)\dd t}\right].
 \label{setup:partition-identity}
\end{align}
The coefficient in its leading exponential is the original
mass reward minus the Gaussian translation cost:
\[
 W\log\frac{4\pi}{W}-W(2\log2-1)
                  =W\left(1+\log\frac\pi W\right).
\]
The prefactor $(W/(4\pi))^{1/2}$ comes from the $-1/2$ in the
power.  These factors cancel exactly, rather than asymptotically,
in \eqref{eq:shifted-expectation}.

Lemma~\ref{ldp:normalizer} proves that the expectation on the
right of \eqref{setup:partition-identity} has logarithm
$o(\gamma^{-2})$.  The partition identity then yields
\begin{equation}
 \lim_{\gamma\downarrow0}\gamma^2\log
 \E\!\left[Z_\gamma(\gamma R_\gamma)^{W/\gamma^2-1/2}\right]
       =W\left(1+\log\frac\pi W\right).
 \label{setup:free-energy}
\end{equation}

\subsection{The field and unit-area measure limits}\label{proof:completion}

The goal is to finish the semiclassical theorem from concentration
under the exact tilted law. For field observables, the reflected
profile and logarithmic mass are continuous at the limiting triple,
so a neighborhood estimate controls every bounded continuous $F$.
For the area measure, tested masses are adjoined by the same
approximation argument as in Lemma~\ref{ldp:joint}; convergence of
these tests and total mass one then give weak convergence without
escape of mass.

\subsubsection{Convergence of the field observables}

We transfer the triple concentration to the field argument in the
exact expectation. Continuity near the limiting triple and boundedness
of $F$ outside that neighborhood give convergence of its expectation.

\begin{proof}[Proof of the field assertion in Theorem~\ref{setup:main}]
\label{ldp:observable-limit}
We first treat bounded continuous functionals on $\widehat E_s$; the
stated $H^{-s}_{\mathrm{loc}}$ conclusion follows by restriction.

The exact ratio \eqref{eq:shifted-expectation} and definition
\eqref{ldp:tilted-law} give the equality
\begin{equation}\label{ldp:exact-observable}
 \E^{(1)}[F(\Phi_\gamma)]
 =\E_{\mathbb Q_\gamma}\left[
 F\left(b_\gamma(\gamma B)+\gamma X^\perp
       -\log Z_\gamma(b_\gamma(\gamma B))\right)\right].
\end{equation}
For $|t|\le T$, the running-supremum inequality gives
\[
 \sup_{|t|\le T}
 \left|b_\gamma(u;t)+2\log\cosh(Wt/2)\right|
 \le3\sup_{|t|\le T}|u(t)|+\frac32\gamma^2T.
\]
Thus the map inside $F$ in \eqref{ldp:exact-observable} approaches
$\phi_W=\log(W/(4\pi))-2\log\cosh(Wt/2)$ whenever
$(u,v,a)$ approaches $(0,0,4\pi/W)$, uniformly for sufficiently
small $\gamma$. Given $\epsilon>0$, continuity of $F$ supplies an
open neighborhood $U$ of that triple such that
\[
 \left|F(b_\gamma(u)+v-\log a)-F(\phi_W)\right|<\epsilon
 \quad\text{on }U
\]
for all sufficiently small $\gamma$. Split
\eqref{ldp:exact-observable} according to this neighborhood:
\begin{align}\label{ldp:F-error}
 \left|\E^{(1)}[F(\Phi_\gamma)]-F(\phi_W)\right|
 &\le\epsilon+2\|F\|_\infty\,
 \mathbb Q_\gamma\left(
 \bigl(\gamma B,\gamma X^\perp,Z_\gamma(b_\gamma(\gamma B))\bigr)
 \notin U\right).
\end{align}
Lemma~\ref{ldp:concentration} makes the second term tend to zero.
Since $\epsilon$ is arbitrary, this proves the expectation limit.

Finally, on each bounded time interval, convergence in
$C_tH_N^{-s}$ implies convergence in $L^2_tH_N^{-s}$. This embeds
continuously into the two-variable $H^{-s}$ space, because
$(1+\xi^2+n^2)^{-s}\le(1+n^2)^{-s}$ in the time-Fourier and
angular-cosine variables. Cutoffs and restriction to compact
subsets of $\Sstrip$ give a continuous map
$\widehat E_s\to H^{-s}_{\mathrm{loc}}(\Sstrip)$. Applying the
first assertion to the composition of $F$ with this map proves
the final assertion.
\end{proof}

\subsubsection{Convergence of the unit-area measure}

We prove weak convergence of the normalized area measure by testing
it against finitely many compactly supported functions, then using
total mass one to control the tails. Under $\mathbb Q_\gamma$,
the unit-area measure has the exact
representation
\begin{equation}\label{ldp:area-representation}
 \frac{e^{b_\gamma(\gamma B;t)}(2\sin\theta)^{-\gamma^2/2}
       :e^{\gamma X^\perp(t,\theta)}:\,dt\,d\theta}
      {Z_\gamma(b_\gamma(\gamma B))}.
\end{equation}
Its identification with the canonical circle-average measure was
proved in the regularization section.

\begin{proof}[Proof of the area assertion in Theorem~\ref{setup:main}]
\label{ldp:area-limit}
The limiting measure to be identified is
\begin{equation}\label{ldp:classical-area}
 \frac W{4\pi}\sech^2(Wt/2)\,dt\,d\theta.
\end{equation}

For finitely many $f\in C_c(\Sstrip)$, adjoin to the triple
\eqref{ldp:triple} the coordinates
\[
 \int_{\Sstrip}f(z)e^{b_\gamma(\gamma B;t)}
 (2\sin\theta)^{-\gamma^2/2}:e^{\gamma X^\perp(z)}:\,d^2z.
\]
The proof of Lemma~\ref{ldp:joint} applies to these coordinates:
the cutoff estimates apply to the positive and negative parts of
$f/\|f\|_\infty$; the finite maps are continuous; and uniform
integrability on bounded energy sets gives deterministic limit
$\int f(z)e^{b_0(u;t)+v(z)}\,d^2z$. The additional coordinates are
bounded in absolute value by $\|f\|_\infty$ times the total mass,
so exponential tightness is retained. The weighted estimates and
uniqueness argument therefore give joint convergence of these
coordinates to $\int f(z)\sech^2(Wt/2)\,d^2z$, while the denominator
in \eqref{ldp:area-representation} tends to $4\pi/W$.

To prevent escape of mass, choose $0\le f\le1$ in $C_c(\Sstrip)$
whose integral against \eqref{ldp:classical-area} exceeds
$1-\epsilon$. The measures in \eqref{ldp:area-representation} have
total mass one; their mass outside $\operatorname{supp}f$ is at
most one minus their integral of $f$, and is therefore less than
$2\epsilon$ with probability tending to one. This establishes
tightness away from both angular boundaries and both horizontal
ends. Here is the corresponding weak-convergence conclusion in explicit
test-function form. Let $g$ be any bounded continuous function on the
open strip, and let $f$ be the compactly supported cutoff just chosen.
The integral of $gf$ converges because $gf\in C_c(\Sstrip)$.
For each of the random and limiting probability measures, the absolute
integral of $g(1-f)$ is at most $\|g\|_\infty$ times its integral
of $1-f$. The latter is at most $2\epsilon$ with probability tending
to one for the random measure and at most $\epsilon$ for the limit.
Sending $\epsilon\downarrow0$ proves convergence of the integral
of $g$ in probability.

To express convergence in the weak topology itself, choose a sequence
$(f_j)_{j\ge1}$ of compactly supported continuous functions dense in
the unit ball of $C_0(\Sstrip)$ for the uniform norm, where $C_0$
denotes continuous functions vanishing at infinity. For probability
measures $\mu,\nu$, use the metric
\[
 d(\mu,\nu)=\sum_{j\ge1}2^{-j}
       \left(1\wedge\left|\int f_jd\mu-\int f_jd\nu\right|\right).
\]
Density shows that convergence in this metric is convergence against
all $C_0$ tests. The cutoff argument in the preceding paragraph,
using total mass one, then upgrades it to weak convergence.
A finite initial sum in this metric tends to zero in probability
by the joint finite-test convergence; the remaining sum is at most
$\sum_{j>J}2^{-j}=2^{-J}$. This proves convergence in probability
in the weak topology. The same test
argument applies to restrictions of continuous functions on
$[-\infty,+\infty]\times[0,\pi]$, and therefore proves weak
convergence on the compactification as well. The limiting density integrates to one and assigns
zero mass to all added boundary points.
\end{proof}

\appendix
\section{The field rate and the Liouville action}
\label{act:section}

The preceding concentration argument also gives a full large-deviation
principle for the normalized field.  Its rate is obtained by minimizing
over the unreflected Brownian input.  We carry out this minimization
explicitly and identify the result with a renormalized Liouville action.

For a function $\varphi$ that belongs to $H^1$ on every bounded full
strip, let
$\overline\varphi(t)=\pi^{-1}\int_0^\pi\varphi(t,\theta)\dd\theta$,
using its locally absolutely continuous representative.  Define
\begin{equation}
\label{act:domain}
 \mathcal D_W^{\mathrm{max}}=
 \left\{\begin{array}{l|l}
  \varphi&
  \varphi-\phi_W\in H^1((-T,T)\times(0,\pi))\text{ for every }T<\infty,\\
  &\displaystyle\int_{\Sstrip}|\nabla(\varphi-\phi_W)|^2\dd z<\infty,
   \quad Z_0(\varphi)=1,\\
  &\overline\varphi(t)\le\overline\varphi(0)\text{ for all }t\in\R
 \end{array}\right\}.
\end{equation}
The mean need only attain its maximum at zero; ties are allowed.  The
energy condition is homogeneous and does not require
$\varphi-\phi_W\in L^2(\Sstrip)$ or limits at the strip ends.

Define the candidate field rate by
\begin{equation}
\label{act:rate}
 \mathcal I(\varphi)=
 \begin{cases}
 \displaystyle\frac1{4\pi}\int_{\Sstrip}|\nabla(\varphi-\phi_W)|^2\dd z
       +W\int_{\Sstrip}(\varphi-\phi_W)e^{\phi_W}\dd z,
                         &\varphi\in\mathcal D_W^{\mathrm{max}},\\[1mm]
 +\infty,                &\text{otherwise}.
 \end{cases}
\end{equation}

\begin{theorem}[The field large-deviation principle]
\label{act:field-ldp}
For every $s>0$, the laws of $\Phi_\gamma$ satisfy a good LDP of speed
$\gamma^{-2}$ in $H^{-s}_{\mathrm{loc}}(\Sstrip)$ with rate
$\mathcal I$.  This rate is nonnegative and vanishes only at $\phi_W$.
\end{theorem}

We first identify the infimum in the field contraction.  For a
finite-energy radial $u$ and $v\in\mathcal H_\perp$, use the $f$ and
$m_\pm$ defined in \eqref{var:reflection-definitions}.

\begin{lemma}
\label{act:contraction}
The normalized field
$\varphi=b_0(u)+v-\log Z_0(b_0(u)+v)$ belongs to
$\mathcal D_W^{\mathrm{max}}$, and
\begin{align}
 &\frac14\int_\R|u'|^2\dd t
   +\frac1{4\pi}\int_{\Sstrip}|\nabla v|^2\dd z
   +\frac{W^2}{4}\int_\R u(t)\sech^2(Wt/2)\dd t
   -W\log\left(\frac W{4\pi}Z_0(b_0(u)+v)\right)
 \notag\\
 &\hspace{20mm}=\mathcal I(\varphi)
                  +W\sum_{\sigma=\pm}m_\sigma(\infty).
 \label{act:contraction-identity}
\end{align}
Every $\varphi\in\mathcal D_W^{\mathrm{max}}$ has such a preimage
with $m_+=m_-=0$.
\end{lemma}

\begin{proof}
The field has area one.  Its vertical mean equals
$b_0(u)-\log Z_0(b_0(u)+v)$, with maximum at zero by
\eqref{var:constraints}.  Also
\[
 \varphi-\phi_W=f+v-
               \log\left(\frac W{4\pi}Z_0(b_0(u)+v)\right).
\]
Orthogonality to the constant angular mode therefore gives
\begin{align*}
 \mathcal I(\varphi)
 &=\frac14\int_\R|f'|^2\dd t
   +\frac1{4\pi}\int_{\Sstrip}|\nabla v|^2\dd z
   +\frac{W^2}{4}\int_\R f(t)\sech^2(Wt/2)\dd t\\
 &\hspace{15mm}-W\log\left(\frac W{4\pi}Z_0(b_0(u)+v)\right).
\end{align*}
The reflection identity proves \eqref{act:contraction-identity}.

Conversely, for $\varphi\in\mathcal D_W^{\mathrm{max}}$ choose
\begin{equation}
\label{act:inverse}
 u(t)=\overline\varphi(t)-\overline\varphi(0)
                         +2\log\cosh(Wt/2),\qquad
 v(t,\theta)=\varphi(t,\theta)-\overline\varphi(t).
\end{equation}
The value at zero gives $u(0)=0$. Since $\phi_W$ is independent of
$\theta$, differentiation of the angular mean and Cauchy--Schwarz give
\[
 \int_\R|u'|^2dt
 =\int_\R|\overline\varphi'-\phi_W'|^2dt
 \le\frac1\pi\int_{\Sstrip}
                  |\partial_t(\varphi-\phi_W)|^2dz<\infty.
\]
The derivative of $v$ is the orthogonal projection of
$\nabla(\varphi-\phi_W)$ onto the zero-angular-mean part, so
$\int|\nabla v|^2\le\int|\nabla(\varphi-\phi_W)|^2$.
Poincar\'e's inequality on every vertical segment gives
$\int|v|^2\le\int|v_\theta|^2$. Thus $v\in\mathcal H_\perp$.  The expression inside either running
maximum is $\overline\varphi(\pm t)-\overline\varphi(0)\le0$, with
value zero at the origin, so the maxima vanish.  Finally
$Z_0(b_0(u)+v)=e^{-\overline\varphi(0)}$ by the area constraint.
The normalized reconstruction consequently equals $\varphi$ exactly.
All preimages have cost at least $\mathcal I(\varphi)$, and this one
attains that value.  Fields outside \eqref{act:domain} have no
finite-energy preimage.
\end{proof}

\begin{proof}[Proof of Theorem~\ref{act:field-ldp}]
The enhanced Gaussian LDP of Lemma~\ref{ldp:joint} has finite rate
only when its mass coordinate is $Z_0(b_0(u)+v)$, with Gaussian cost
$\frac14\int|u'|^2+\frac1{4\pi}\int|\nabla v|^2$.
The weighted closed-set upper bound and matching open-set lower bound
of Lemma~\ref{ldp:weighted-upper}, divided by the denominator whose
rate is zero by \eqref{ldp:normalizer-eq}, prove the full tilted LDP.
Its rate is the left side of
\eqref{act:contraction-identity} on the mass graph and infinity
elsewhere.  Coercivity and Lemma~\ref{var:compactness} show it is good.

For positive mass, the map
\[
 (u,v,a)\longmapsto b_0(u)+v-\log a
\]
is continuous into
$\widehat E_s=C_{\mathrm{loc}}(\R;H^{-s}_N(0,\pi))$, now including
the constant angular mode.  To handle its logarithm at zero, replace
$\log a$ by $\log(a\vee1/n)$.  These maps are continuous everywhere.
On each tilted-rate sublevel the mass is bounded away from zero:
otherwise compactness would give a finite-rate point with zero mass,
contradicting positivity of $Z_0(b_0(u)+v)$.  In particular,
\[
 \inf\left\{\text{tilted rate at }(u,v,a):0\le a\le1/n\right\}
                 \longrightarrow+\infty.
\]
Otherwise bounded-rate points with masses tending to zero would have
a convergent subsequence in a compact rate sublevel, with mass zero.
Apply the tilted LDP upper bound to the closed set $a\le1/n$.
The probability on which the logarithm cutoff changes the field
therefore has logarithmic upper rate tending to $-\infty$.
  Extended contraction applies, and
Lemma~\ref{act:contraction} identifies its rate as \eqref{act:rate}.

The true reconstruction in \eqref{eq:shifted-expectation} uses
$b_\gamma$.  Adding $\gamma^2t/2$ changes the running maximum on
$[0,t]$ by an amount between zero and $\gamma^2t/2$.  It follows that
\[
 |b_\gamma(u;t)-b_0(u;t)|\le\gamma^2|t|/2,
\]
uniformly in $u$.  Hence the true and limiting reconstructions are
exponentially equivalent in $\widehat E_s$.

Restriction from $\widehat E_s$ to $H^{-s}_{\mathrm{loc}}(\Sstrip)$
is continuous for every $s>0$.  Indeed, after even angular extension
and a compactly supported horizontal cutoff, Plancherel's theorem and
\[
 (1+\xi^2+n^2)^{-s}\le(1+n^2)^{-s}
\]
bound the isotropic negative Sobolev norm by the
$L^2_tH^{-s}_N$ norm, which in turn is bounded by the local
$C_tH^{-s}_N$ norm.  Interior angular cutoffs are bounded Sobolev
multipliers.  The restriction may identify general distributions
supported on the horizontal boundary.  It is injective on finite-rate
fields, however, since those are $H^1$ functions on bounded full
strips and are determined by their interior distributions.  Goodness
ensures attainment of each finite contraction infimum.  Ordinary
contraction therefore gives exactly \eqref{act:rate} in the stated
topology.  Nonnegativity and uniqueness of the zero follow from
Lemma~\ref{var:selection} and the minimizing preimage
\eqref{act:inverse}.
\end{proof}

\subsection{Renormalization at the strip ends}

The gradient of $\phi_W$ tends to $\mp W$ at the two ends, so the
unnormalized Dirichlet energy diverges.  The following cutoff expression
gives the action with its boundary contributions and exact constant.
It is defined for every $\varphi-\phi_W$ of finite relative Dirichlet
energy, without an area or embedding constraint:
\begin{equation}
\label{act:counterterms}
 \mathscr S_W(\varphi)=\lim_{T\to\infty}\left\{
  \frac1{4\pi}\int_{[-T,T]\times(0,\pi)}|\nabla\varphi|^2\dd z
  +\frac W2\bigl(\overline\varphi(T)+\overline\varphi(-T)\bigr)
  +\frac{W^2}{2}T\right\}.
\end{equation}

\begin{lemma}
\label{act:action}
The limit \eqref{act:counterterms} exists and is finite.  For
$\varphi\in\mathcal D_W^{\mathrm{max}}$,
\begin{equation}
\label{act:action-rate}
 \mathcal I(\varphi)=\mathscr S_W(\varphi)-\mathscr S_W(\phi_W)
 =\mathscr S_W(\varphi)+W\bigl(1+\log(\pi/W)\bigr).
\end{equation}
For general finite-relative-energy $\varphi$, the scale-invariant
relative action is
\begin{align}
 &\mathscr S_W(\varphi)-\mathscr S_W(\phi_W)-W\log Z_0(\varphi)
 \notag\\
 &\quad=\frac1{4\pi}\int_{\Sstrip}|\nabla(\varphi-\phi_W)|^2\dd z
       +W\int_{\Sstrip}(\varphi-\phi_W)e^{\phi_W}\dd z
       -W\log\int_{\Sstrip}e^{\varphi-\phi_W}e^{\phi_W}\dd z.
 \label{act:relative-action}
\end{align}
\end{lemma}

\begin{proof}
Put $h=\varphi-\phi_W$.  Its angular mean obeys
$|\overline h(t)|\le|\overline h(0)|+
\sqrt{|t|/\pi}\,\|\nabla h\|_2$.
Its zero-mean lateral part belongs to $H^1(\Sstrip)$ by Poincar\'e.
These bounds show $h\in L^2(e^{\phi_W}\dd z)$, so
Lemma~\ref{var:onofri} proves finiteness and positivity of
$Z_0(\varphi)$ as well.

Expanding the square and integrating by parts in the angular mean
gives
\begin{align*}
 &\frac1{4\pi}\int_{[-T,T]\times(0,\pi)}
                    (|\nabla\varphi|^2-|\nabla\phi_W|^2)\dd z\\
 &\quad=\frac1{4\pi}\int_{[-T,T]\times(0,\pi)}|\nabla h|^2\dd z
 -\frac W2\tanh(WT/2)\bigl(\overline h(T)+\overline h(-T)\bigr)\\
 &\qquad+\frac{W^2}{4}\int_{-T}^{T}\overline h(t)\sech^2(Wt/2)\dd t.
\end{align*}
After adding the difference of the boundary terms in
\eqref{act:counterterms}, the remaining boundary contribution is
$\frac W2(1-\tanh(WT/2))(\overline h(T)+\overline h(-T))$, which
tends to zero by the square-root growth bound.  The other terms
converge to
$\frac1{4\pi}\int|\nabla h|^2+W\int h e^{\phi_W}$.
Direct integration of the baseline gives
\[
 \frac1{4\pi}\int_{[-T,T]\times(0,\pi)}|\nabla\phi_W|^2\dd z
    =\frac{W^2}{2}T-W\tanh(WT/2),\qquad
 \phi_W(T)=-WT+\log(W/\pi)+o(1).
\]
Therefore $\mathscr S_W(\phi_W)=W(\log(W/\pi)-1)$, and all stated
identities follow.  This computation requires no convergence of $h$
or its derivative at either end.
\end{proof}

\subsection{Embedding and the classical equation}

The domain restriction in \eqref{act:rate} is essential.  The translated
profile $\phi_W(t+a,\theta)$ has zero relative action for every
$a\in\R$, but when $a\ne0$ its angular mean is maximized at $-a$.
Its rate in the maximum embedding is therefore infinite.  In general
the action is invariant under horizontal translation, whereas the rate
in fixed coordinates includes the stated embedding constraint.

For completeness, this invariance can be checked without imposing
endpoint limits. Fix a translation $a\in\R$ and put
$h=\varphi-\phi_W$,
$k(t)=\phi_W(t+a)-\phi_W(t)$. The translated relative perturbation is
$h(t+a,\theta)+k(t)$. Expanding its energy and subtracting the energy
of $h$ leaves the mixed term and the energy of $k$. The mixed term is
\[
 \frac1{2\pi}\int_{\Sstrip}\nabla h(t+a,\theta)\cdot\nabla k(t)dz
 =W\int_{\Sstrip}h(t+a,\theta)
                  \bigl(e^{\phi_W(t+a)}-e^{\phi_W(t)}\bigr)dz.
\]
Indeed, $\Delta k=-2\pi W(e^{\phi_W(t+a)}-e^{\phi_W(t)})$,
and integration by parts has no angular boundary term because $k$
is radial. Its horizontal boundary terms vanish since $k'$ decays
exponentially and the angular mean of $h$ grows at most as
$1+\sqrt{|t|}$. By the substitution $s=t+a$, the first term on the
right is $W\int h(s,\theta)e^{\phi_W(s)}ds\,d\theta$.
Consequently this mixed term cancels the change in the linear source
$W\int h e^{\phi_W}$. The remaining difference is
\[
 \frac1{4\pi}\int_{\Sstrip}|\nabla k|^2dz
                       +W\int_{\Sstrip}k e^{\phi_W}dz=0.
\]
The equality follows from radial equality in Onofri, since
$\int e^{k+\phi_W}=\int e^{\phi_W(t+a)}=1$.
Finally $Z_0(\varphi(\,\cdot+a))=Z_0(\varphi)$ by change of
variables. This proves invariance of the action and its scale-invariant
version in \eqref{act:relative-action}.

The Euler--Lagrange equation refers to this invariant action before
fixing the embedding, or equivalently to its variational problem modulo
translations.  For a smooth stationary point, varying
$\mathscr S_W(\varphi)-W\log Z_0(\varphi)$ by a smooth function $\eta$ compactly
supported in the horizontal direction, but unrestricted on the angular
boundary, gives
\[
 \frac1{2\pi}\int_{\Sstrip}\nabla\varphi\cdot\nabla\eta\dd z
       -\frac{W}{Z_0(\varphi)}\int_{\Sstrip}\eta e^\varphi\dd z=0.
\]
For tests supported in the interior, integration by parts yields
$\Delta\varphi+2\pi W e^\varphi/Z_0(\varphi)=0$.
Allowing arbitrary boundary traces of $\eta$ then forces
$\partial_n\varphi=0$, where $n$ is the outward unit normal.
After area normalization this becomes
\begin{equation}
\label{act:liouville}
 \Delta\varphi+2\pi W e^\varphi=0\quad\text{in }\Sstrip,
 \qquad \partial_n\varphi=0\quad\text{on its horizontal boundaries}.
\end{equation}
The boundary condition is natural for smooth stationary points, not
an additional trace condition on generic finite-rate fields.  It
corresponds to curvature $\pi W$ for the metric $e^\varphi|\dd z|^2$.
The explicit field $\phi_W$ solves this equation, and the action minimum
is attained uniquely there in the maximum embedding.

\end{document}